\documentclass [a4paper,reqno, 11pt]{amsart}
\pdfoutput=1

\usepackage{fullpage}

\usepackage{graphicx} 

\usepackage{amssymb,amsthm,amsmath,mathrsfs,amsfonts}

\usepackage[indent,skip=.2\baselineskip]{parskip}

\usepackage[dvipsnames]{xcolor}

\usepackage[bookmarks, bookmarksdepth=2, colorlinks=true, linkcolor=blue!85!black, citecolor=blue!85!black, urlcolor=magenta!85!black]{hyperref}

\usepackage[thinc]{esdiff}

\usepackage{tikz-cd}

\usepackage[perpage]{footmisc}

\usepackage{booktabs}

\usepackage{url}

\usepackage{verbatim}

\usepackage[nameinlink]{cleveref}
\theoremstyle{plain}
\newtheorem{theorem}{Theorem}[section]

\newtheorem{proposition}[theorem]{Proposition}
\newtheorem{lemma}[theorem]{Lemma}

\theoremstyle{definition}

\newtheorem{example}[theorem]{Example}

\crefname{ex}{Example}{Examples}
\crefname{thm}{Theorem}{Theorems} 
\crefname{lem}{Lemma}{Lemmas}
\crefname{prop}{Proposition}{Propositions}
\crefname{cor}{Corollary}{Corollaries} 
\crefname{con}{Conjecture}{Conjectures} 
\crefname{def}{Definition}{Definitions}
\crefname{rmk}{Remark}{Remarks}

\numberwithin{equation}{section}

\renewcommand{\AA}{\mathbb{A}}
\newcommand{\PP}{\mathbb{P}}
\newcommand{\CC}{\mathbb{C}}
\newcommand{\RR}{\mathbb{R}}

\newcommand{\NN}{\mathbb{N}}
\newcommand{\ZZ}{\mathbb{Z}}

\newcommand{\IG}{\mathrm{IG}}
\newcommand{\gen}{\mathrm{gen}}

\newcommand{\I}{\mathcal{I}} 
\newcommand{\V}{\mathcal{V}} 

\newcommand{\M}{\mathcal{M}} 
\newcommand{\calS}{\mathcal{S}} 

\newcommand{\calD}{\mathcal{D}}
\newcommand{\calA}{\mathcal{A}}
\newcommand{\calH}{\mathcal{H}}
\newcommand{\calB}{\mathcal{B}}

\DeclareMathOperator{\Sec}{Sec} 

\newcommand{\E}{\mathcal{E}} 

\DeclareMathOperator{\Pic}{Pic} 

\newcommand{\from}{\colon}

\AddToHook{cmd/bfseries/after}{\boldmath}
\AddToHook{cmd/normalfont/after}{\unboldmath}

\usepackage{enumitem}
\setlist[enumerate,itemize]{leftmargin=3em}

\title{Moment varieties of the inverse~Gaussian and gamma~distributions are nondefective}
\author{Oskar Henriksson, Kristian Ranestad, Lisa Seccia, Teresa Yu}

\begin{document}

\begin{abstract}
    We show that the parameters of a $k$-mixture of inverse Gaussian or gamma distributions are algebraically identifiable from the first $3k-1$ moments, and rationally identifiable from the first $3k+2$ moments. Our proofs are based on Terracini's classification of defective surfaces, careful analysis of the intersection theory of moment varieties, and a recent result on sufficient conditions for rational identifiability of secant varieties by Massarenti--Mella.
\end{abstract}

\maketitle

\vspace{-0.5em}

\section{Introduction}
Secant varieties have played a central role in algebraic geometry since the turn of the 20th century, and have recently also been used to shed light on problems in many areas of applied mathematics, including tensor decomposition \cite{BCCGO18}, rigidity theory \cite{CMNT23}, and optimization \cite{OT24}. 

In this work, we use secant varieties to study the \textit{method of moments} for parameter estimation in statistics. Given a stochastic variable $X$ of a distribution depending on  parameters $\theta=(\theta_1,\ldots,\theta_n)$, the moments $m_r(\theta)=\mathbb{E}[X^r]$ for $r\in\NN$ are often rational functions of the parameters (see, e.g., \cite{BS15}). In its simplest form, the goal of the method of moments is to estimate the parameters by solving the system 
\begin{equation}\label{eq:moment_equations}
m_r(\theta)=\widehat{m}_r\quad\text{for $r=1,\ldots,d$},
\end{equation}
for some number $d$ of sample moments $\widehat{m}_1,\ldots,\widehat{m}_d$ computed from a sample of $X$. Geometrically, this corresponds to  studying the fibers of the rational map
\begin{equation}\label{eq:moment_parametrization}
\CC^n\dashrightarrow\PP^d,\quad \theta\mapsto [1:m_1(\theta):\cdots:m_d(\theta)],
\end{equation}
where the Zariski closure $\M_d$ of the image is the $d$th \emph{moment variety} of the distribution. 

Of fundamental statistical importance is the following identifiability problem: How many moments $d$ do we expect to need to include in system \eqref{eq:moment_equations} for it to have finitely many or even a unique solution?  We say that we have \emph{algebraic identifiability} if the fibers of $\CC^n\dashrightarrow\M_d$ are finite over generic points of $\M_d$ (which is equivalent to $\dim(\M_d)=n$), and we say that we have \emph{rational identifiability} if fibers over generic points of $\M_d$ contain a unique point.

These types of identifiability questions have been studied for many distributions using algebraic techniques. The most well-understood examples come from Gaussian distributions \cite{AFS16,ARS18,AAR21,LAR25,Blo25, BCMO23}, and other examples include uniform distributions on polytopes \cite{KSS20}, Dirac and Pareto distributions \cite{GKW20}, as well as inverse Gaussian distributions and gamma distributions \cite{HSY24}.

The connection between moment varieties and secant varieties is the following. Consider a given distribution with $n$ parameters and $d$th moment variety $\M_d$. Then one can consider a $k$-mixture of this distribution; it has $kn+k-1$ parameters, and its $d$th moment variety is the secant variety $\Sec_k(\M_d)$. 
This perspective has previously been exploited in the Gaussian case to study the method of moments for their mixtures, starting with the seminal paper \cite{ARS18}. In this paper, we extend and generalize some of those techniques to the gamma and inverse Gaussian distributions, whose moment varieties have more complicated structures \cite{HSY24}.

Our first main result concerns algebraic identifiability of $k$-mixtures. We have algebraic identifiability from the first $d$ moments if and only if 
\[\dim(\Sec_k(\M_d))=kn+k-1.\]
We approach this through the notion of \emph{non-defectivity} in the theory of secant varieties, where a variety $X\subseteq\PP^d$ is said to be \emph{k-nondefective} if $\dim(\Sec_k(X))=\min\{k\dim(X)+k-1,d\}$, and \emph{$k$-defective} otherwise.

\begin{theorem}\label{thm:intro_algebraic_identifiability}
    Let $\M_d$ be the $d$th moment variety for the inverse Gaussian or gamma distribution. Then $\M_d$ is $k$-nondefective for each $d\geq 2$ and $k\geq 2$, in the sense that $$\dim(\Sec_k(\M_d))=\min\{3k-1,d\}.$$
    In particular, we have algebraic identifiability from the first $3k-1$ moments for $k$-mixtures of the inverse Gaussian distribution and $k$-mixtures of the gamma distribution.
\end{theorem}

The defectivity of secant varieties is a well-studied yet difficult problem in algebraic geometry. In particular, the classification of defective Segre--Veronese varieties is a classical problem, with many applications to computer science and statistics, due to such secant varieties being closely related to symmetric tensor decomposition \cite{AH95,Lan12,ABGO24}.
Although moment varieties are typically not Segre--Veronese, they are often still determinantal varieties. For example, 
 \cite{AFS16} showed that the $d$th Gaussian moment variety in $\PP^d$ is given by the maximal minors of
\[H_1=\begin{pmatrix}
0 & x_0 & 2 x_1 & 3 x_2  & \cdots & (d-1) x_{d-2} \\
x_0 & x_1 & x_2 & x_3 & \cdots & x_{d-1} \\
x_1 & x_2 & x_3 & x_4 & \cdots & x_d
\end{pmatrix},\]
and it was shown in \cite{HSY24} that the $d$th moment variety for the inverse Gaussian in $\PP^d$ is defined by the maximal minors of
\[H_2=\begin{pmatrix}
    x_0^2 & x_0 & x_1 & x_2 & x_3 & \cdots & x_{d-2} \\
    0 & x_1 & 3x_2 & 5x_3 & 7x_4 & \cdots & (2d-3)x_{d-1}\\
    x_1^2 & x_2 & x_3 & x_4 & x_5 & \cdots & x_d
\end{pmatrix}.\]

Our strategy for proving \Cref{thm:intro_algebraic_identifiability} is similar to that used in \cite{ARS18} to study $k$-mixtures of univariate Gaussian distributions, in that it builds on Terracini's classification of defective surfaces, and uses intersection theory to rule out each of the possibilities for defectivity. However, having only linear entries in the matrix $H_1$ is essential in the argument of \cite{ARS18}, as their approach relies on the intersection theory of a smooth surface defined using a Hilbert--Burch matrix. As the matrix $H_2$ has nonlinear entries, we must develop a more general approach.

With \Cref{thm:intro_algebraic_identifiability} in place, we turn our attention to rational identifiability. For $k$-mixtures, rational identifiability up to permutation of the mixture components corresponds to proving that a generic point on $\Sec_k(\M_d)$ lies on a unique $k$-secant. In the theory of secant varieties, this property is commonly referred to as \emph{$k$-identifiability} of $\M_d$. Sufficient conditions for $k$-identifiability in terms of non-defectivity was developed in \cite{CM23}, and was recently sharpened to conditions that also involve the geometry of the Gauss map in \cite{MM24}. These conditions have been used to show identifiability results in, e.g., rigidity theory \cite{CMNT23}, Waring theory \cite{CP24}, and for various Gaussian mixture distributions \cite{LAR25,Blo25,BCMO23}. We use it to prove the following result, which is known to be true in the Gaussian case from \cite{LAR25}.

\begin{theorem}\label{thm:intro_rational_identifiability}
    Up to permuting the mixture components, we have rational identifiability from the first $3k+2$ moments for $k$-mixtures of the inverse Gaussian distribution and $k$-mixtures of the gamma distribution for any $k\geq 2$.
\end{theorem}

\subsection*{Future research directions}
Our result on algebraic identifiability is optimal in the sense that it is impossible to have algebraic identifiability from fewer than $3k-1$ moments for dimension reasons. However, for rational identifiability it is still an open question whether it is possible to get rational identifiability from fewer than $3k+2$ moments. Based on numerical experiments, it is conjectured in \cite{LAR25,HSY24} that $3k$ moments is enough in the Gaussian, gamma and inverse Gaussian cases.
Another major challenge for the future would be to find effective techniques for proving non-defectivity for determinantal varieties of higher dimensions, where there is no obvious known analog of Terracini's classification.

\subsection*{Organization of the paper}
The remainder of the paper is organized as follows. 
In \Cref{sec:prelim}, we provide background on the intersection theory needed for our results, and outline our proof strategy for \Cref{thm:intro_algebraic_identifiability}. In \Cref{sec:inverse_gaussian,sec:gamma}, we carry out this strategy for the inverse Gaussian and gamma distributions respectively. In \Cref{sec:rational_identifiability}, we deduce \Cref{thm:intro_rational_identifiability} on rational identifiability of $k$-mixtures. 

\subsection*{Acknowledgements}
The authors are grateful to Elisenda Feliu, Francesco Galuppi, Alexandros Grosdos Koutsoumpelias, Ragni Piene, and Jose Israel Rodriguez for helpful discussions. The authors also thank Bernd Sturmfels and the Max Plank Institute for Mathematics in the Sciences for hosting OH, KR and TY, and providing a productive visit, during which part of this research was performed. Part of this research was also performed while OH, LS, and TY were visiting the Institute for Mathematical and Statistical Innovation (IMSI) for the Algebraic Statistics long program, which was supported by NSF grant DMS-1929348. 

LS was partially supported by SNSF grant TMPFP2-217223. TY was partially supported by NSF grant DGE-2241144.
OH was partially supported by the Novo Nordisk Foundation grant NNF20OC0065582, as well as the European Union under the Grant Agreement~no.~101044561, POSALG. Views and opinions expressed are those of the authors only and do not necessarily reflect those of the European Union or European Research Council (ERC). Neither the European Union nor ERC can be held responsible for them.

\section{Preliminaries}\label{sec:prelim}
\subsection{Moment varieties of mixtures}

Real-world datasets often display multimodal behavior or underlying heterogeneity, suggesting that they consist of different subpopulations or patterns. To effectively model such complexity, statisticians often employ \textit{mixtures of a distribution}, which are convex combinations of a given probability distribution.
 
In this paper, we focus on moment varieties of mixtures. 
Consider a univariate distribution with parameter space $\Theta\subseteq\RR^n$ and probability density function $p\from\RR\times\Theta\to\RR$, such that the first $d$ moments $m_1(\theta),\ldots,m_d(\theta)$ depend rationally on $\theta\in\Theta$ and parametrize the $d$th moment variety $\M_d\subseteq\PP^d$.
Then the \emph{$k$-mixture} of the distribution has parameter space $\Theta^k\times\Delta_{k-1}$ where $\Delta_{k-1}\subseteq\RR^k$ is the $(k-1)$-dimensional probability simplex, and its density function is given by 
\[f\from\RR\times\Theta^k\times\Delta_{k-1}\to\RR,\quad (x,\theta^{(1)},\ldots,\theta^{(k)},\alpha)\mapsto\sum_{i=1}^{k}\alpha_i p(x,\theta^{(i)}).\]
The $d$th moment variety of the $k$-mixture is the Zariski closure of the image of the rational map
\begin{equation}\label{eq:parametrization_of_mixture_moment_variety}
(\CC^n)^k\times\V\!\left(\textstyle\sum_{i=1}^k\alpha_i-1\right)\dashrightarrow\PP^d,\:\: (\theta^{(1)},\ldots,\theta^{(k)},\alpha)\mapsto \left[\textstyle\sum_{i=1}^{k}\alpha_im_r(\theta^{(i)})\right]_{r=0,\ldots,d},
\end{equation}
where $\V(\cdot)$ denotes the zero locus of a polynomial.
Equivalently, the moment variety of the $k$-mixture is the $k$th secant variety $\Sec_k(\M_d)$.  Proving that we have algebraic identifiability from the first $d$ moments corresponds to proving that 
\begin{equation}
\label{eq:dimension_that_gives_algebraic_identifiability_for_Md}\dim\!\left(\Sec_k (\M_d)\right) = kn+k-1,\end{equation}
while proving that we have rational identifiability up to the natural label-swapping action of the symmetric group $S_k$ on $(\CC^n)^k\times\V(\sum_{i=1}^k\alpha_i-1)$ corresponds to proving that a generic point in $\Sec_k(\M_d)$ lies on a unique $k$-secant. 

Previous work on identifiability for mixture distributions has focused on the Gaussian distribution. In \cite{ARS18}, the authors prove algebraic identifiability from the first $3k-1$ moments, and in \cite{LAR25}, the authors prove rational identifiability from the first $3k+2$ moments.

In what follows, we will study two other distributions that play an important role in statistics: the \emph{inverse Gaussian distribution}, and the \emph{gamma distribution}. Both these distributions are two-dimensional (in the sense that $n=2$), and we will use results from the theory of secant varieties of surfaces to prove algebraic identifiability from $3k-1$ moments, and rational identifiability from the $3k+2$ moments. 

\subsection{Proof strategy}\label{sec:strategy}
We now outline our strategy for proving \Cref{thm:intro_algebraic_identifiability} on algebraic identifiability.
Recall that for a variety $X\subseteq\PP^d$, it holds that 
$$\dim(\Sec_k(X))\leq\min\{k\dim(X)+(k-1),d\},$$
where the upper bound in the right-hand side is called the \emph{expected dimension} of $\Sec_k(X)$. The variety $X$ is said to be \emph{$k$-nondefective} if the bound is attained, and \emph{$k$-defective} otherwise. 

We will show \Cref{thm:intro_algebraic_identifiability} by proving that $\M_d$ is $k$-nondefective for all $k\geq 2$ and $d\geq 2$ via the following classical classification result due to Terracini.
Here, we use the formulation from \cite[Theorem~8]{ARS18} (see also \cite[Theorem~1.3]{CC02}). 

\begin{theorem}[Terracini's classification]\label{thm:terracini}
Let $X \subseteq \mathbb{P}^d$ be a reduced, irreducible, nondegenerate projective surface. If $X$ is $k$-defective, then $k \geq 2$ and one of the following two possibilities hold:
\begin{enumerate}
    \item[(1)] $X$ is the quadratic Veronese embedding of a rational normal surface in $\PP^k$ of degree $k-1$.
    \item[(2)] $X$ is contained in a cone over a curve, with apex a linear space of dimension at most $k-2$.
\end{enumerate}
Furthermore, for general points $p_1, \ldots , p_k$ on $X$ there is a hyperplane section tangent along
a curve $C$ that passes through these points. In case (1), the curve $C$ is irreducible; in case
(2), the curve $C$ decomposes into $k$ algebraically equivalent curves $C_1, \ldots, C_k$ with $p_i \in C_i$.
\end{theorem}

For both the inverse Gaussian and gamma distributions, we can rule out case (1) in the Terracini classification based on information about the singular loci. If case (1) were to hold, then $X$ would either be smooth, or singular at only one point, but neither of these hold for our moment varieties \cite{HSY24}. In order to rule out (2), the general strategy will be as follows.

The starting point is to turn the defining parametrization of $\M_d$ by the moments $m_i(\theta)$ into a rational parametrization $\phi\from\PP^2\dashrightarrow\M_d$ by homogeneous forms $f_i(\theta,s)$ with finitely many indeterminacy points $P_1,\ldots,P_r$, and then form a smooth resolution $\pi\from\calS_d\to\PP^2$ of the locus of indeterminacy, such that the parametrization lifts to a morphism $\tilde{\phi}\from\calS_d\to\M_d$ that makes the following diagram commute: 
\begin{equation}\label{eq:resolution_commutative_diagram}   
\begin{tikzcd}
	{\calS_d} \\
	{\PP^2} & {\M_d}\rlap{\,.}
	\arrow["\phi", dashed, from=2-1, to=2-2]
	\arrow["\pi"', from=1-1, to=2-1]
	\arrow["{\tilde{\phi}}", from=1-1, to=2-2]
\end{tikzcd}
\end{equation}
Once we have constructed $\calS_d$, we will use intersection-theoretic calculations in the  Picard group $\Pic(\calS_d)$ of Weil divisors modulo linear equivalence to rule out case (2).

A class that will play a particularly important role is the class $H$ of the strict transform of curves $\V(f_\gen)$, where $f_\gen$ is a generic linear combination of the coordinates of the parametrization $\phi=[f_0:\cdots:f_d]$. This class coincides with the class of the pullback of hyperplane sections of $\M_d \subseteq \PP^d$ via $\tilde{\phi}\from \calS_d\to\M_d$. For this class, case (2) in the Terracini classification has the following consequence.

\begin{lemma}[Lemma 10, \cite{ARS18}]\label[lem]{lem:divisors_Di_in_case_2}
Suppose that 
$\M_d$ satisfies condition (2) in \Cref{thm:terracini}. Then, for any $k$ general points $x_1,\ldots,x_k\in \calS_d$, there exist linearly equivalent effective divisors $\calD_1\ni x_1,\ldots, \calD_k\ni x_k$, and a hyperplane section of $\M_d$, with pullback $\calH$ to $\calS_d$, such that 
\begin{equation}\label{eq:definition_of_A}
\calA=\calH-2\calD_1-\cdots -2\calD_k
\end{equation}
is an effective divisor on $\calS_d$.
\end{lemma}

The strategy is to show that the existence of such divisors $\calA$ and $\calD_i$ gives rise to a contradiction. We will proceed by casework: the $\calD_i$ divisors are linearly equivalent and therefore their projections in $\PP^2$ all have the same degree $a\geq 1$. We will use intersection-theoretic calculations in $\Pic(\calS_d)$ to derive a contradiction for any choice of $a$. 

We note that the class of the $\calD_i$ on $\calS_d$ is \emph{moving}; for us, this means that for any fixed point on the surface, there is a linearly equivalent curve that goes through this point. We will often utilize this to reduce to smaller values of $a$, or to obtain contradictions for a given choice of $a$.

Another way to rule out large values of $a$ is to work with as large of a $k$ as possible for which $\M_d$ is $k$-defective. Indeed, \Cref{lem:divisors_Di_in_case_2} tells us that $\M_d$ being $k$-defective would imply that the value of $\deg(f_\gen)-2ka$ must be positive in order for the divisor $\calA$ to be effective. Since $\deg(f_\gen)$ is fixed, assuming $k$ to be large allows us to reduce to small values of $a$. An important tool to this end is the following lemma.

\begin{lemma}[{Corollary 7, \cite{ARS18}}]
\label[lem]{lem:lower_bound_on_d_for_defective_surfaces}
If a surface $X\subseteq\PP^d$ is $k$-defective for some $k\geq 2$, then there is a $k'\geq(d-2)/3$ such that $X$ is $k'$-defective.
\end{lemma}

After reducing to small values of $a$, we are in some cases able to infer that $d\leq 8$, in which case non-defectivity can be checked computationally, as summarized by the following lemma. 

\begin{lemma}\label{lem:base_cases}
    For both the inverse Gaussian or gamma distribution, the $d$th moment variety $\M_d$ is $k$-nondefective for all $2\leq d\leq 8$ and all $k\geq 2$.
\end{lemma}

\begin{proof}
This was verified by evaluating the Jacobian of the parametrization \eqref{eq:parametrization_of_mixture_moment_variety} of $\Sec_k(\M_d)$ for $k\leq \lceil \frac{d+1}{3}\rceil$ at random integer values of the distribution parameters, and then computing the rank using exact rational arithmetic\footnote{Explicit choices of parameters and the associated Jacobians that certify these computations can be found in the repository \url{https://github.com/oskarhenriksson/moment-varieties-inverse-gaussian-and-gamma}.}. See also the discussion in  \cite[Section~5.1]{HSY24}.
\end{proof}

With algebraic identifiability in place, rational identifiability follows from an argument based on the conditions given in \cite{MM24}. This is the subject of \Cref{sec:rational_identifiability}.

\subsection{Construction and properties of the resolution}
We end the section by describing the general structure of the construction of the resolution $\calS_d$, and some key facts about the structure of $\Pic(\calS_d)$ that are common for both distributions. In \Cref{sec:inverse_gaussian,sec:gamma}, we will describe the particularities of this construction for each of the respective distributions. 

The resolution will consist of a sequence of blowups
\[\calS_d=\calS_{r,\ell_r}\to\cdots \to \calS_{i,j}\to\cdots\to\PP^2\]
where we, for the $i$th indeterminacy point $P_i$, construct a sequence of $\ell_i$ blowups in the following way (where $\ell_i$ depends on $i$ and the distribution at hand):
\begin{itemize}
    \item We start by blowing up the intermediate surface $\calS_{i-1,\ell_{i-1}}$ obtained in the previous step (or $\PP^2$, when $i=1$) at $P_{i,0}=P_i$. Let $\calS_{i,1}$ denote the resulting blowup, $\E_{i,1}$ the exceptional divisor, and $\phi_{i,1}\from\calS_{i,1}\dashrightarrow\M_d$ the lift of the previous map. This map $\phi_{i,1}$ turns out to have a unique indeterminacy point $P_{i,1}$ on $\E_{i,1}$.
    \item We blow up $\calS_{i,1}$ at $P_{i,1}$. Let $\calS_{i,2}$ be the resulting blowup,  $\E_{i,2}$ the exceptional divisor, and $\phi_{i,2}\from\calS_{1,2}\dashrightarrow\M_d$ the lift of $\phi_{i,1}$, with unique indeterminacy point $P_{i,2}$ on $\E_{i,2}$.
    \item Continue in this way, until we blow up $\calS_{i,\ell_i-1}$ at $P_{i,\ell_i-1}$, to obtain a surface $\calS_{i,\ell_i}$ with exceptional divisor $\E_{i,\ell_i}$, and a lift $\phi_{i,\ell_i}\from\calS_{i,\ell_i}\dashrightarrow\M_d$, such that the map $\phi_{i,\ell_i}$ has no further indeterminacy points on $\E_{i,\ell_i}$.
\end{itemize}
In the $j$th step for the $i$th indeterminacy point, we construct $\calS_{i,j}$ by picking an affine chart $\AA^2$ around $P_{i,j-1}$, which we blow up to $\calB_{i,j}$, and we then define $\calS_{i,j}$ as the Zariski closure of $\calB_{i,j}$ in $\calS_{i,j-1}\times\PP^1$, and $\pi_{i,j}\from\calS_{i,j}\to\calS_{i,j-1}$ as the extension of the blowup map $\calB_{i,j}\to\AA^2$, so that we get a commutative diagram
\[\begin{tikzcd}
	{\calB_{i,j}} & {\calS_{i,j}}\rlap{\quad$\subseteq\quad \calS_{i,j-1}\times\PP^1$ } \\
	{\AA^2} & {\calS_{i,j-1}} & {\M_d.}
	\arrow[hook, from=1-1, to=1-2]
	\arrow[from=1-1, to=2-1]
	\arrow["{\pi_{i,j}}"', from=1-2, to=2-2]
	\arrow["{\phi_{i,j}}", dashed, from=1-2, to=2-3]
	\arrow[hook, from=2-1, to=2-2]
	\arrow["\qquad{\phi_{i,j-1}}\qquad"', dashed, from=2-2, to=2-3]
\end{tikzcd}\]
We finally let $\calS_d=\calS_{r,\ell_r}$, and $\tilde{\phi}=\phi_{r,\ell_r}$, and take $\pi\from\calS_d\to\PP^2$ to be the composition of all maps $\pi_{i,j}$, giving the commutative diagram \eqref{eq:resolution_commutative_diagram}.
It follows from the general theory of blowups at points in the projective plane (see, e.g., \cite[§V.3]{Har77}) that $\Pic(\calS_d)$ is a free abelian group generated by the classes $E_{i,j}$ of the exceptional divisors $\E_{i,j}$ obtained in the construction of $\calS_d$, as well as the class $L$ of a line in $\PP^2$ pulled back to $\calS_d$. Furthermore, the intersection number pairing $\cdot\from \Pic(\calS_d)\times\Pic(\calS_d)\to\ZZ$ is diagonal with
\begin{equation}\label{eq:picard_grup_intersections}
    L^2=1,\qquad E_{i,j}^2=-1\:\:\text{for all}\:\:i=1,\ldots,r\:\:\text{and}\:\: j=1,\ldots,\ell_i.
\end{equation} 
It furthermore follows from \cite[Proposition~3.6]{Har77} that the class $\tilde{C}$ of the strict transform to $\calS_d$ of any irreducible curve $\mathcal{C}$ in $\PP^2$ can be expressed in terms of these generators as
\begin{equation}\label{eq:class_of_strict_transform}
\tilde{C}=\deg(C)L - \sum_{i=1}^r\sum_{j=1}^{\ell_i} m_{i,j}E_{i,j},
\end{equation}
where $m_{i,j}$ is the multiplicity at $P_{i,j-1}$ of the strict transform of $C$ on $\calS_{i,j-1}$.

\section{The inverse Gaussian distribution}\label{sec:inverse_gaussian}
The inverse Gaussian distribution has two parameters $\mu$ and $\lambda$, and its $d$th moment variety $\M_d^\IG\subseteq\PP^d$ is a surface that is the Zariski closure of the image of the map 
$$(\CC^*)\times\CC\to\PP^d,\quad (\mu,\lambda)\mapsto [m_0:\cdots:m_d],$$ 
where the moments are defined recursively as
\begin{equation}\label{eq:parametrization_inverse_gaussian}
m_0=1,\qquad m_1=\mu,\qquad m_i=\tfrac{2i-3}{\lambda}\mu^2m_{i-1} + \mu^2m_{i-2}\:\:\:\text{for}\:\:\: i\geq 2.
\end{equation}
Note that it follows directly from the recursive formula that for $i>0$,
$$m_i=\frac{\mu^i\,p_{i-1}(\lambda,\mu)}{\lambda^{i-1}},$$
where $p_{i}(\lambda,\mu)$ is the homogenization of the degree-$i$ Bessel polynomial (see, e.g., \cite[§1]{Gro51} for a definition), with $\mu$ as the homogenization variable. For proving non-defectivity, we will use the following basic  algebraic and geometric properties of $\M_d^\IG$ as a starting point.

\begin{theorem}[\S3, \cite{HSY24}]\label{thm:IG_results_from_previous_paper}
    Let $d \geq 3$. The homogeneous ideal $\I(\M_d^\IG)$ is generated by $\binom{d-1}{3}$ cubics and $\binom{d-1}{2}$ quartics, given by the maximal minors of the $(3\times d)$-matrix
    $$\begin{pmatrix}
     x_0^2 & x_0 & x_1 & x_2 & x_3&\cdots& x_{d-2}\\
     0 & x_1 & 3x_2 & 5x_3 & 7x_4 &\cdots& (2d-3)x_{d-1}\\
     x_1^2 & x_2 & x_3 &x_4  &x_5& \cdots& x_d
    \end{pmatrix}.$$
    Furthermore, $\mathcal{M}_{d}^\IG$ has degree $(d-1)^2$. The singular locus of $\M_d^\IG$ is given by the line $x_0=x_1=\ldots=x_{d-2}=0$ and the point $x_1=x_2= \ldots=x_d=0$ in $\PP^d$.
\end{theorem}

Our main goal in this section is to rule out case (2) in the Terracini classification, using the strategy outlined in \Cref{sec:strategy}. 
We begin by homogenizing and clearing denominators in the parametrization \eqref{eq:parametrization_inverse_gaussian}, which gives the following rational map,
\begin{align}\label{eq:equations-Fgen}
    \phi\from\PP^2\dashrightarrow\M_d^\IG,\quad [\lambda:\mu:s]\mapsto [f_0(\lambda,\mu,s):f_1(\lambda,\mu,s):\cdots:f_d(\lambda,\mu,s)],
\end{align}
where the coordinate functions are given by 
\begin{align*}
    f_0=\lambda^{d-1}s^d,\quad
    f_1=\lambda^{d-1}s^{d-1} \mu,\quad
    \ldots\quad 
    f_r=\lambda^{d-r}s^{d-r}\mu^r p_{r-1}(\lambda,\mu),\quad 
    \ldots\quad
    f_d=\mu^{d} p_{d-1}(\lambda,\mu).
\end{align*}
The locus of indeterminacy consists of the following points:
\begin{align*}
    P_1=[0\mathbin{:}0\mathbin{:}1],\quad P_2=[1\mathbin{:}0\mathbin{:}0],\quad P_3=[x_1\mathbin{:}1\mathbin{:}0],\quad \ldots\quad P_{d+1}=[x_{d-1}\mathbin{:}1\mathbin{:}0],
\end{align*}
where $x_1,\ldots,x_{d-1}$ are the distinct roots of the Bessel polynomial $p_{d-1}(\lambda,1)$. Note that it follows by \cite[Theorem~1]{Gro51} that all roots of each Bessel polynomial are simple. 
Let $f_\gen$ be a generic combination of the $d+1$ coordinate functions of $\phi$ in \eqref{eq:equations-Fgen}, and consider the curve $\V(f_\gen)\subseteq\PP^2$.  

As described in \Cref{sec:strategy}, we construct $\pi\from\calS_d\to\PP^2$ by a sequence of blowups. In the case for the inverse Gaussian distribution, we end up needing $\ell_1=d+1$, $\ell_2=\cdots=\ell_{d+1}=1$ blowup steps at the respective indeterminacy points. That this suffices to resolve the indeterminacy locus is proven by the following lemmas. The intersection-theoretic implications of these lemmas that we will use in the rest of this section are collected in \Cref{lem:strict_transforms_IG}. We also provide an example of some key steps of the construction in the $d=4$ case in \Cref{ex:ig_blowups_d_4}. 

For ease of notation, we will write $\E_i=\E_{i,1}$ for all $i\geq 2$ throughout this section. Note that since the indeterminacy points are isolated points in $\PP^2$ and the blowup at a point is birational outside that point, it suffices to independently describe the sequence of blowups over each $P_i$.

\begin{lemma}\label[lem]{lem:blowup_ig_P1}
Let $f_\gen=\sum_{k=0}^{d} a_k f_k$ be a linear combination of the coordinate functions of $\phi$ with general coefficients $(a_0,\ldots,a_d)\in\CC^{d+1}$. Then the following holds:
\begin{enumerate}
    \item $P_1$ is a zero of $f_\gen$ with multiplicity $d-1$. 
    \item The exceptional divisor $\mathcal{E}_{1,1}$ intersects the strict transform of $\V(f_\gen)$ at a single point $P_{1,1}$ with multiplicity $d-1$; this point corresponds to the tangent direction $\lambda=0$ at $P_1$.
    \item Fix $j\in\{2,\ldots,d\}$, and suppose we have already blown up at $P_{1,1},\ldots,P_{1,j-1}$ to obtain $\calS_{1,j}\to\PP^2$. Then the lift $\phi_{1,j}\from\calS_{1,j}\dashrightarrow\PP^d$ has a single new indeterminacy point $P_{1,j}$ on the exceptional divisor $\E_{1,j}$. It is a point on the strict transform of $\V(f_\gen)$ with multiplicity one.
    \item Consider the blowup $\calS_{1,d+1}$ at $P_{1,d}$. Then the lift $\phi_{1,d+1}\from\calS_{1,d+1}\dashrightarrow\PP^d$ has no new indeterminacy points on $\E_{1,d+1}$.
\end{enumerate}
\end{lemma}

\begin{proof}
\textbf{Part (1):} Consider the affine chart $\PP^2\cap\{s=1\}\cong\AA^2_{(\lambda,\mu)}$, so $P_1$ is the origin in this chart and the coordinate functions of $\phi$ are given by
\[f_0=\lambda^{d-1},\qquad f_j=\lambda^{d-j}\mu^j p_{j-1}(\lambda,\mu),\quad \text{for }j=1,\ldots,d.\]
The lowest degree terms in $\lambda,\mu$ of these functions are $\lambda^{d-1},\lambda^{d-1}\mu,\ldots,\mu^d$ respectively, and so we see that $P_1$ is a zero of $f_\gen$ with multiplicity $d-1$.

\smallskip
\noindent
\textbf{Part (2):} We continue to work in the affine chart $\AA^2_{(\lambda,\mu)}$. The resulting blowup at $P_1$ is locally given by the coordinates
\[\calB_{1,1}=\{((\lambda,\mu),[v_1:w_1])\in\AA^2\times\PP^1:\lambda w_1=\mu v_1\},\]
with blowup morphism given by projection onto $\AA^2$. 
The exceptional divisor $\mathcal{E}_{1,1}\subseteq\calB_{1,1}$ is given by $\{(0,0)\}\times\PP^1$,
and the strict transform of the line $\V(\lambda)$ is given by $\{((0,\mu),[0:1]):\mu\in\CC\}$. Recall that $\phi_{1,1}\from\calS_{1,1}\dashrightarrow\PP^d$ denotes the lift of $\phi$.

Consider the affine chart $\calB_{1,1}\cap \{v_1=1\}\cong\AA^2_{(\lambda,w_1)}$. The restriction to $\AA^2_{(\lambda,w_1)}$ of $\phi_{1,1}$ after factoring out the common factor $\lambda^{d-1}$ is given by
\[(\lambda,w_1)\mapsto \left[\begin{array}{l}1:\lambda w_1:\cdots\\[0.2em]
\lambda w_1^j p_{j-1}(\lambda,\lambda w_1):\cdots\\[0.2em]
\lambda w_1^d p_{d-1}(\lambda,\lambda w_1)\end{array}\right].\]
These coordinate functions have no common zeros.

Now consider the affine chart $\calB_{1,1}\cap\{w_1=1\}\cong\AA^2_{(\mu,v_1)}$. The restriction of $\phi_{1,1}$ to this chart after factoring out the common factor $\mu^{d-1}$ is given by
\[(\mu,v_1)\mapsto \left[\begin{array}{l}v_1^{d-1}:\mu v_1^{d-1}:\cdots\\[0.2em]
\mu v_1^{d-j}p_{j-1}(\mu v_1,\mu):\cdots\\[0.2em]
\mu \,p_{d-1}(\mu v_1,\mu)\end{array}\right].\]
The coordinate functions have a common zero of multiplicity $d-1$ at $(\mu,v_1)=(0,0)$. In the coordinates of $\calB_{1,1}$, this is the point $P_{1,1}=((0,0),[0:1])$, which lies on both the exceptional divisor $\mathcal{E}_{1,1}$ and the strict transform of $\V(\lambda)$. In particular, we see that claim (2) holds.

\smallskip
\noindent
\textbf{Parts (3) and (4):} 
\textit{Blowup at $P_{1,1}$}: We first consider the blowup of $\calS_{1,1}$ at $P_{1,1}$, which has local coordinates
\[\calB_{1,2}=\{((\mu,v_1),[v_2:w_2])\in\AA^2\times\PP^1:\mu w_2=v_1v_2\}.\]

Consider the affine chart $\calB_{1,2}\cap\{w_2=1\}\cong \AA^2_{(v_1,v_2)}$. Then the restriction of the lift $\phi_{1,2}$ to this chart, after factoring out the common factor $v_1^{d-1}$, has the zeroth coordinate function given by $1$, and so the coordinate functions have no common zeros.

Now consider the affine chart $\calB_{1,2}\cap\{v_2=1\}\cong \AA^2_{(\mu,w_2)}$. The restriction of $\phi_{1,2}$ to this chart, after factoring out the common factor $\mu^{d-1}$, is given by
\[(\mu,w_2)\mapsto \left[\begin{array}{l}w_2^{d-1}:\mu w_2^{d-1}:\cdots\\[0.2em]
\mu w_2^{d-j}p_{j-1}(\mu w_2,1):\cdots\\[0.2em]
\mu \,p_{d-1}(\mu w_2,1)\end{array}\right].\]
To factor out $\mu^{d-1}$, we use the fact that $p_j(\mu^2w_2,\mu)=\mu\,p_j(\mu w_2,1)$ for each $j=1,\ldots,d-1$. Observe that there is a common zero $(\mu,w_2)=(0,0)$ of multiplicity one, and this corresponds to the indeterminacy point $P_{1,2}=((0,0),[1:0])\in\calB_{1,2}$ on $\mathcal{E}_{1,2}$.

\smallskip
\noindent
\textit{Blowup at $P_{1,2}$}: Now we blow up $\calS_{1,2}$ at $P_{1,2}$ to obtain $\calS_{1,3}\to\PP^2$, with local coordinates
\[\calB_{1,3}=\{((\mu,w_2),[v_3:w_3]:\mu w_3=w_2 v_3\},\]
and lift $\phi_{1,3}$ of $\phi$.
 
The restriction of $\phi_{1,3}$ to the affine chart $\calB_{1,3}\cap \{v_3=1\}\cong\AA^2_{(\mu,w_3)}$, after factoring out the common factor $\mu$, has zeroth coordinate function given by $\mu^{d-2} w_3^{d-1}$, and $d$th coordinate function $p_{d-1}(\mu^2 w_3,1)$, which has a degree zero term. Thus, there are no common zeros in this chart.

Now consider the affine chart $\calB_{1,3}\cap \{w_3=1\}=\AA^2_{(w_2,v_3)}$. The restriction of $\phi_{1,3}$ to this chart, after factoring out the common factor $w_2$, has coordinate functions
\[(w_2,v_3)\mapsto \left[\begin{array}{l}w_2^{d-2}:v_3 w_2^{d-1}:\cdots\\[0.2em]
v_3 w_2^{d-k}p_{k-1}(w_2^2 v_3,1):\cdots\\[0.2em]
v_3 \,p_{d-1}(w_2^2 v_3,1)\end{array}\right].\]
There is a common zero $(w_2,v_3)=(0,0)$ of multiplicity one, and this corresponds to the indeterminacy point $P_{1,3}=((0,0),[0:1])\in\calB_{1,3}$ on the exceptional divisor $\E_{1,3}$.

\smallskip
\noindent
\textit{Blowup at $P_{1,j}$ for $j=3,\ldots,d$}: Now suppose by induction that we have already blown up at $P_{1,j-1}$ to obtain the surface $\calS_{1,j}\to\PP^2$, and that there is a indeterminacy point $P_{1,j}$ given by $(0,0)\in\AA^2_{(w_2,v_j)}\cong \calB_{1,j}\cap\{w_{j}=1\}$. Blowup at $P_{1,j}$ to obtain $\calS_{1,j+1}\to\PP^2$, with local coordinates
\[\calB_{1,j+1}=\{((w_2,v_{j}),[v_{j+1}:w_{j+1}])\in\AA^2\times\PP^1:w_2w_{j+1}=v_{j}v_{j+1}\}.\]
Let $\phi_{1,j+1}\from\calS_{1,j+1}\dashrightarrow\PP^d$ denote the lift of $\phi$.

Consider the affine chart $\calB_{1,j+1}\cap\{v_{j+1}=1\}\cong\AA^2_{(w_2,w_{j+1})}$. The restriction of $\phi_{1,j+1}$ to this chart, after factoring out the common factor $w_2$, has coordinate functions
\[(w_2,w_{j+1})\mapsto \left[\begin{array}{l}w_2^{d-j}:w_{j+1} w_2^{d-1}:\cdots\\[0.2em]
 w_{j+1}w_2^{d-k}p_{k-1}(w_2^{j}w_{j+1},1):\cdots\\[0.2em]
w_{j+1} \,p_{d-1}(w_2^{j}w_{j+1},1)\end{array}\right].\]
If $j=d$, then there are no common zeros, as the zeroth coordinate function is $1$. If $j=3,\ldots,d-1$, then there is a common zero $(w_2,w_{j+1})=(0,0)$ of multiplicity one, and this corresponds to the indeterminacy point $P_{1,j+1}=((0,0),[0:1])\in\calB_{1,j+1}$ on the exceptional divisor $\E_{1,j+1}$ of the blowup.

In the other affine chart $\calB_{1,j+1}\cap\{w_{j+1}=1\}\cong\AA^2_{(v_{j},v_{j+1})}$, the restriction of $\phi_{1,j+1}$ after factoring out the common factor $v_{j}$ does not have any common zeros. As in the case of blowing up at $P_{1,2}$, the zeroth coordinate function is $v_{j}^{d-j} v_{j+1}^{d-j-1}$, while the $d$th coordinate function is $p_{d-1}(v_{j}^{j} v_{j+1}^{j-1},1)$, which has a degree zero term.

This completes the proof of parts (3) and (4).
\end{proof}

\begin{lemma}\label[lem]{lem:blowup_ig_other_base_points}
    Let $f_\gen=\sum_{j=0}^{d} a_j f_j$ be a linear combination of the coordinate functions $f_j$ for $\phi$ with general coefficients $(a_0,\ldots,a_d)\in\CC^{d+1}$.
    \begin{enumerate}
        \item $P_2$ is a zero of $f_\gen$ with multiplicity $d$. Furthermore, the lift $\phi_2$ of $\phi$ to the blowup $\calS_{2,1}\to\PP^2$ at $P_2$ has no new indeterminacy points on the exceptional divisor $\mathcal{E}_2$.
        \item For each $i=3,\ldots,d+1$, the point $P_i\in\PP^2$ is a zero of $f_\gen$ with multiplicity $1$. Furthermore, the lift $\phi_{i}$ of $\phi$ to the blowup $\calS_{i,1}\to\PP^2$ at $P_i$ has no new indeterminacy points on the exceptional divisor $\E_i$.
    \end{enumerate}
\end{lemma}

\begin{proof}
\textbf{Part (1):} Consider the affine chart $\PP^2\cap\{\lambda=1\}\cong\AA^2_{(\mu,s)}$. In this chart, $P_2$ is the origin of $\AA^2$, and the coordinate functions of $\phi$ are given by
\[f_j(\mu,s)=s^{d-j}\mu^j p_{j-1}(1,\mu). \]
The lowest degree terms in $\mu,s$ of these functions are all of degree $d$, and so this proves part (1). Furthermore, these degree $d$ terms of $f_0$ and $f_d$ are $s^d$ and $\mu^d$ respectively, and so the coordinate functions do not have a common tangent direction at the origin. Thus, the strict transform of $\V(f_\gen)$ does not intersect the exceptional divisor $\mathcal{E}_2$.

\smallskip
\noindent
\textbf{Part (2):} Recall that $P_i=[x_{i-2}:1:0]$ where $x_{i-2}$ is a simple root of the Bessel polynomial $p_{d-1}(\lambda,1)$. Consider the affine chart $\PP^2\cap\{\mu=1\}=\AA^2_{(\lambda,s)}$, along with the change of coordinates $\lambda'=\lambda-x_{i-2}$. In this chart, the coordinate functions of $\phi$ are given by
\[f_0=(\lambda'+x_{i-2})^{d-1} s^d,\qquad f_j=(\lambda'+x_{i-2})^{d-j}s^{d-j}p_{j-1}(\lambda'+x_{i-2},1),\quad \text{for }j=1,\ldots,d.\]
The lowest degree terms in $\lambda',s$ of these functions are $s^d,s^{d-1},\ldots,s,\lambda'$ respectively. Thus, $P_i$ is a zero of $f_\gen$ with multiplicity $1$, and the coordinate functions do not have a common tangent direction at the origin $(\lambda',s)=(0,0)$. This completes the proof.
\end{proof}

For illustration purposes, we now explicitly carry out the blowups over $P_1$ in the case  $d=4$.

\begin{example}\label[ex]{ex:ig_blowups_d_4}
\textbf{Blowing up at $P_{1}$:} Consider the affine chart $\PP^2\cap\{s=1\}\cong\AA^2$ with coordinates $(\lambda,\mu)$, and where $\phi$ is given by 
\[(\lambda,\mu)\mapsto 
\left(\lambda^{3},\:
\lambda^{3} \mu,\: 
\lambda^{2} \mu^2 (\lambda+ \mu),\: 
\lambda\mu^{3}(\lambda^2+3\lambda \mu+3\mu^2),\:
\mu^{4} (15\mu^3+15\mu^2 \lambda+6\mu \lambda^2+\lambda^3)\right).
\]
Blowing up at $P_1$ corresponds to blowing up at $(0,0)$ in this chart. The resulting surface $\calS_{1,1}$ is the closure in $\PP^2\times\PP^1$ of 
\[\calB_{1,1}=\{((\lambda,\mu),[v_1:w_1])\in\AA^2\times\PP^1:\lambda w_1=\mu v_1\}\subseteq\PP^2\times\PP^1.\]

Consider the affine chart $\calB_{1,1}\cap \{w_1=1\}$.
Substituting $\lambda=\mu v_1$ and factoring out a common factor $\mu^3$ gives that $\phi_{1,1}$ on this chart is given by
\[(\mu,v_1)\mapsto \left[\begin{array}{l}v_1^3:\\
    \mu v_1^3:\\
    \mu v_1^2  (\mu v_1+ \mu):\\
    \mu v_1  (\mu^2 v_1^2+3\mu^2 v_1 +3\mu^2):\\
    \mu (\mu^3v_1^3+6\mu^3v_1^2+15\mu^3 v_1+15\mu^3)\end{array}\right].\]
Note that the exceptional divisor $\E_{1,1}$ of the blowup is defined by $\mu=0$, and it intersects the strict transform of $\V(f_\gen)$ with multiplicity $3$ at the point $(0,0)$. This corresponds to the point $P_{1,1}\in\calS_{1,1}$, which gives the tangent direction $\lambda=0$ of $\V(f_\gen)$ at $P_1$.

\smallskip
\noindent
\textbf{Blowing up at $P_{1,1}$:} The surface $\calS_{1,2}$ is defined as the closure in $\calS_{1,1}\times\PP^1$ of
\[\calB_{1,2}=\{((\mu,v_1),[v_2:w_2])\in\AA^2\times\PP^1:\mu w_2=v_1v_2\}.\]
Taking the affine chart $\{v_2=1\}$, substituting $v_1=\mu w_2$ and factoring out $\mu^3$, we get that $\phi_{1,2}$  on this chart is given by 
\[(\mu,w_2)\mapsto\left[\begin{array}{l}
w_2^3:\\
w_{2}^{3}\mu:\\
w_2^2  \mu (\mu w_2+ 1):\\
w_2  \mu(\mu^2 w_2^2+3\mu w_2 +3):\\
\mu(\mu^3 w_2^3+6\mu^2 w_2^2+15\mu w_2+15)
\end{array}\right],\]
with indeterminacy point $(0,0)$, with multiplicity $1$; this corresponds to the point $P_{1,2}\in\calS_{1,2}$.

\smallskip
\noindent
\textbf{Blowing up at $P_{1,2}$:}
We blow up at this point to get a surface $\calS_{1,3}$ with local coordinates $(\mu,w_2)\times[v_3:w_3]$. 
Considering the affine chart $\{w_3=1\}$, substituting $\mu=w_2v_3$ into our coordinate functions and factoring out $w_2$, we see that $\phi_{1,3}$ on this chart is given by
\[(w_2,v_3)\mapsto\left[\begin{array}{l}
 w_{2}^{2}:\\
 w_{2}^{3} v_3:\\
 w_2^2 v_3  (w_{2}^{2} v_{3}+ 1):\\
w_2  v_3 (w_{2}^{4} v_{3}^{2} +3 w_{2}^{2} v_{3}  +3):\\
v_3 (w_2^6v_3^3+6w_2^4v_3^2+15w_2^2v_3+15).
\end{array}\right],\]
with indeterminacy point $(0,0)$ with multiplicity $1$; we call this point $P_{1,3}\in\calS_{1,3}$.

\smallskip
\noindent
\textbf{Blowing up at $P_{1,3}$:}
We blow up at this point to get a surface $\calS_{1,4}$ with local coordinates $(w_2,v_3)\times[v_4:w_4]$. 
Considering the affine chart $\{v_4=1\}$, substituting $v_3=w_2w_4$ into our coordinate functions, and factoring out $w_2$, we obtain
\[(w_2,w_4)\mapsto \left[\begin{array}{l}
w_{2}\\
w_{2}^{3} w_4\\
w_2^2 w_4  (w_{2}^{3} w_4+ 1)\\
w_2w_4 (w_{2}^{6} w_4^{2} +3 w_{2}^{3} w_4  +3)\\
w_4 (w_2^9w_4^3+6w_2^6w_4^2+15w_2^3w_4+15).
\end{array}\right]\]
with indeterminacy point $(0,0)$ of multiplicity $1$; we call this point $P_{1,4}\in\calS_{1,4}$.

\smallskip
\noindent
\textbf{Blowing up at $P_{1,4}$:}
We blow up at this point to get a surface $\calS_{1,5}$ with local coordinates $(w_2,w_4)\times[v_5:w_5]$. 
Considering the affine chart $\{v_5=1\}$, substituting $w_4=w_2w_5$ into our coordinate functions, and factoring out $w_2$, we finally obtain
\[(w_2,w_5)\mapsto \left[\begin{array}{l}
1\\
w_{2}^{3} w_5\\
w_2^2 w_5  (w_{2}^{4} w_5+ 1)\\
w_2w_5 (w_{2}^{8} w_5^{2} +3 w_{2}^{4} w_5  +3)\\
w_5 (w_2^{12}w_5^3+6w_2^8w_5^2+15w_2^4w_5+15).
\end{array}\right],\]
which lacks indeterminacy points. One can check that at each of the blowup steps above, the other choice of affine chart also does not contain any indeterminacy points for the lift of $\phi$. The lift of $\phi$ is therefore now well-defined over the original point $P_1\in\PP^2$.
\end{example}

The following intersection-theoretic formulas are a direct consequence of \Cref{lem:blowup_ig_P1} and \Cref{lem:blowup_ig_other_base_points} together with \eqref{eq:class_of_strict_transform}.

{\samepage
\begin{lemma}\label[lem]{lem:strict_transforms_IG}
The following formulas hold in $\Pic(\calS_d)$:
\begin{enumerate}
    \item Let $\mathcal{L}_1\subseteq\PP^2$ be the line through $P_1$ and $P_2$. The class of the strict transform of $\mathcal{L}_1$ to $\calS_d$ is $L-E_{1,1}-E_2$.
    \item Let $\mathcal{L}_2\subseteq\PP^2$ be the line through $P_1$ and with tangent direction $\lambda=0$ (meaning that the strict transform of $\mathcal{L}_2$ under the initial blowup map goes through the point $P_{1,1}$). Then the class of the strict transform of $\mathcal{L}_2$ is $L-E_{1,1}-E_{1,2}$. 
    \item Let $\mathcal{L}_3\subseteq\PP^2$ be the line going through the colinear points $P_3,\ldots,P_{d+1}$. The class of its strict transform in $\Pic(\calS_d)$ is $L-E_3-\cdots -E_{d+1}$.
    \item\label{it:H_expression_IG} The class of the strict transform of $\V(f_\gen)$ to $\calS_d$ for $f_\gen=\sum_{i=0}^d a_if_i$ with generic coefficients $(a_0,\ldots,a_d)\in\CC^{d+1}$ is
    $$H=(2d-1)L-(d-1)E_{1,1}-(d-1)E_{1,2}-E_{1,3}-\cdots-E_{1,d+1}-dE_{2}-E_3-\cdots -E_{d+1}.$$
\end{enumerate}
\end{lemma}
}

As a consequence of \Cref{lem:blowup_ig_P1,lem:blowup_ig_other_base_points,lem:strict_transforms_IG} we get the following proposition.

\begin{proposition}
     The map $\tilde{\phi}\from\calS_d\to\M_d^{\IG}$ is a birational morphism.
\end{proposition}
\begin{proof} On the one hand, by \Cref{lem:blowup_ig_P1,lem:blowup_ig_other_base_points}, the map $\tilde{\phi}:=\phi_d$ has no indeterminacy points, i.e., it is a morphism.
 On the other hand, by part (\ref{it:H_expression_IG}) of \Cref{lem:strict_transforms_IG}, we have that
\[H^2=(2d-1)^2-(d-1)^2-(d-1)^2-(d-1)-d^2-(d-1)=(d-1)^2=\deg(\M_d^{\IG}),\]
where the last equality follows from \Cref{thm:IG_results_from_previous_paper}. 
Since $H$ is the pullback of a hyperplane section of $\M_d^{\IG}$ along $\tilde{\phi}$, we have that  $H^2=\deg(\M_d)\deg(\tilde{\phi})$. Therefore, we conclude that $\deg(\tilde{\phi})=1$. Since $\tilde{\phi}$ is dominant, the desired result follows.
\end{proof}

\begin{theorem}\label{thm:IG_nondef}
The moment variety $\M_d^{\IG}$ is $k$-nondefective for all $k\geq 2$ and $d\geq 2$.
\end{theorem}

\begin{proof}
Fix $d\geq 2$, and suppose for contradiction that $\M_d^{\IG}$ is $k$-defective; we may assume that $3k+2\geq d$ via \Cref{lem:lower_bound_on_d_for_defective_surfaces}. By \Cref{thm:terracini} and \Cref{thm:IG_results_from_previous_paper}, it must be that $\M_d^{\IG}$ is contained in a cone over a curve, i.e., case (2) of Terracini's classification holds, as the singular locus of $\M^{\IG}_d$ is a line and this cannot occur in case (1). Construct the smooth resolution  $\pi\from\calS_d\to\PP^2$ of the indeterminacy locus of $\phi\from\PP^2\dashrightarrow\M^{\IG}_d$ via the sequence of blowups as detailed earlier in this section. Recall that $H$ denotes the class of the linear system on $\calS_d$ representing hyperplane sections of $\M_d^{\IG}\subseteq\PP^d$, pulled back to $\calS_d$ via $\tilde{\phi}\from\calS_d\to\M^{\IG}_d$. 
Then by \Cref{lem:divisors_Di_in_case_2}, there exists an effective divisor $\calA$ on $\calS_d$ with class $A\in\Pic(\calS_d)$ and linearly equivalent divisors $\calD_1,\ldots,\calD_k$ with class $D\in\Pic(\calS_d)$ such that $H$ can be expressed as
\[H=A+2kD.\]
The images of the $\calD_i$'s in $\PP^2$ all have the same degree $a\geq 1$. Our proof is organized by analyzing the different possibilities for this degree $a$, and in each case we derive a contradiction. We are then able to conclude that (2) of \Cref{thm:terracini} is not possible, and so $\M_d^{\IG}$ is not $k$-defective. 

Note that a generic enough $\calH$ (the pullback of a hyperplane section of $\M_d^{\IG}$ via \Cref{lem:divisors_Di_in_case_2}) will not contain any of the exceptional divisors. This implies that none of the curves $\calD_i$'s will contain any exceptional divisors, and so $D$ can be represented by a strict transform of a curve in $\PP^2$. By \eqref{eq:class_of_strict_transform}, it then follows that 
\begin{equation}\label{eq:D-expansion}
D= aL - b_1 E_{1,1} -b_2 E_2 -b_3 E_{1,2} - \sum \limits_{i=1}^{d-1} c_i E_{1,i+2} -\sum \limits_{i=1}^{d-1} c'_i E_{i+2},
\end{equation}
where $a= D \cdot L$ is a positive integer giving the degree of the representative's image in $\PP^2$, and $b_1,b_2,b_3,c_1, \ldots, c_{d-1}, c'_1, \ldots, c'_{d-1}$ are nonnegative.
Then
$$0\leq L \cdot A= L \cdot H -2k D \cdot L= (2d-1)-2ka.$$
From this inequality, and using \Cref{lem:lower_bound_on_d_for_defective_surfaces}, we have $(2ka+1)/2\leq d \leq 3k+2$,
which implies
\begin{equation}\label{inequality:k}
     ka+\frac{1}{2}\leq d \leq 3k+2.
\end{equation}

We now proceed by casework on the possibilities for $a\geq 1$. The cases of $a\geq 4$ and $a=3$ are  straightforward consequences from the inequality \eqref{inequality:k}, while the cases of $a=2$ and $a=1$ require a more careful analysis of the coordinate functions of $\phi$.

\smallskip
\noindent
\textbf{The case $a\geq 4$:}
The inequality \eqref{inequality:k} becomes 
$$  4k+ \dfrac{1}{2}\leq d \leq 3k+2$$
which is a contradiction, since $k\geq 2$.

\smallskip
\noindent
\textbf{The case $a=3$:} The inequality \eqref{inequality:k} becomes
$$  3k+ \dfrac{1}{2}\leq d \leq 3k+2.$$
Thus, we have only two possibilities for $d$: either $d=3k+1$ or $d=3k+2$. 

Assume $d=3k+1$. Then, we have
\begin{align*}
    A&=(2d-1-6k)L-(d-1-2k b_1) E_{1,1}-(d-2k b_2) E_2-(d-1-2k b_3) E_{1,2}-\cdots\\
    &=L-(3k-2k b_1) E_{1,1}-(3k+1-2k b_2) E_2-(3k-2k b_3) E_{1,2}-\cdots.
\end{align*}
Since $\calA$ is effective and $L\cdot A=1$, the curve $\pi(\calA)\subseteq\PP^2$ is a line.  Therefore we must have $3k-2kb_1\leq 1$, $3k+1-2k b_2\leq 1$, and $3k-2kb_3\leq 1$. Since $k \geq 2$, we get $b_1,b_2,b_3 \geq 2$. 
We see that $(L-E_{1,1}-E_2)\cdot D<0$, so the strict transform of the line $\mathcal{L}_{1}$ is a (fixed) component of the $\calD_i$'s. This means that the residual parts $\pi(\calD_i)\setminus\mathcal{L}_1$ of the curves $\pi(\calD_i)$ are moving conics. We can then reduce to the case $a=2$, treated below.

If $d=3k+2$, we obtain
\begin{align*}
    A&=(2d-1-6k)L-(d-1-2k b_1) E_{1,1}-(d-2k b_2) E_2-(d-1-2k b_3) E_{1,2}-\cdots\\
    &=3L-(3k+1-2k b_1) E_{1,1}-(3k+2-2k b_2) E_2-(3k+1-2k b_3) E_{1,2}-\cdots.
\end{align*}
Again, since $\calA$ is effective and $L\cdot A=3$, the curve $\pi(\calA)\subseteq\PP^2$ is a cubic. It will pass through the points $P_1, P_2$ with multiplicity given by the coefficients above, as well as pass through $P_1$ with tangent direction $\lambda=0$. For degree reasons, we must have that
\[3k+1-2k b_1\leq 3,\quad 3k+2-2k b_2\leq 3,\quad 3k+1-2k b_3\leq 3,\] 
which implies $b_1,b_3\geq 1$ and $b_2\geq 2$. If $b_1>1$ or $b_3>1$, we could reduce to the case $a=2$ as above (the plane cubic $\pi(\calD_i)$ would need to contain a fixed line between $P_1$ and $P_2$). If $b_1=b_3=1$ we obtain $3k+1-2k\leq 3$, which gives $k\leq 2$. Hence, we have reduced to the case when $k=2$ and $d=3k+2=8$, which is covered by \Cref{lem:base_cases}.

\smallskip
\noindent
\textbf{The case $a=2$:} For degree reasons, we can assume that $b_i \leq 1$; otherwise we can reduce to the case $a=1$. In particular, if $b_i=2$ for some $i$, then $\pi(\calD_i)$ is a conic with a double point, i.e., the union of two lines through the double point. Since $D$ is moving, one of these two lines must also be moving. We could then reduce to the $a=1$ case.

Furthermore, we can assume $b_1=b_2=b_3=1$, since otherwise we get $d\leq 8$ (which is covered by \Cref{lem:base_cases}). For example, if $b_1=0$, we obtain
$A= (2d-1-4k)L-(d-1)E_1-\cdots$, which implies $d-1\leq 2d-1-4k\leq 2d-1-4(d-2)/3$, where in the last inequality we used \eqref{inequality:k}.
Geometrically, this means that each curve $\calD_i$ intersects the exceptional divisors $\E_{1,1}$, $\E_2$ and $\E_{1,2}$.
In this case, we have that $\pi(\calD_i)\subseteq\PP^2$ is given by a quadric of the form 
$$g_i(\lambda,\mu,s)=\alpha_i s\lambda+\beta_i \mu^2+\gamma_i \lambda \mu.$$
In fact, since $\pi(\calD_i)$ passes through $P_1$ and $P_2$, the monomials $s^2$ and $\lambda^2$ cannot appear in $g_i$, and since $b_3=1$, the tangent direction at $P_1$ must be $\{\lambda=0\}$, so the monomial $s\mu$ also cannot appear. 

It follows from \eqref{eq:definition_of_A} that there exists some linear combination $f_\gen=\sum_{i=0}^d a_if_i$ of the coordinate functions with generic coefficients, and a nonzero homogeneous polynomial $h(\lambda,\mu,s)$ corresponding to the plane curve $\pi(\calA)$ such that
\[f_\gen=h(\lambda,\mu,s)\cdot \prod_{i=1}^k \left( \alpha_i s\lambda+\beta_i \mu^2+\gamma_i \lambda \mu\right)^2\] 
where $(\alpha_i, \beta_i, \gamma_i)\neq (0,0,0)$ for each $i$.
The contradiction will now follow from a divisibility argument, that relies on the following key observation: The monomials of $f_r$ for $r>0$ are $\lambda^{d-1-i}\mu^{r+i}s^{d-r}$ for $i\in\{0,\ldots,r-1\}$. Hence, the sets of monomials of $f_0,\ldots,f_d$ are pairwise disjoint, and are of the form $\lambda^\ell s^m \mu^n$ with $\ell+m+n=2d-1$, $\ell \leq d-1, m\leq d, n \leq 2d-1$. Moreover, we have either $m \leq \ell$, or $m=d$ and $\ell=d-1$.

We may reduce to the case where at least one $\alpha_i$ is nonzero; otherwise, each $\pi(\calD_i)$ is reducible and consists of two lines, and we can reduce to the case $a=1$. If some $\alpha_i$ is nonzero, then we may assume that all the $\alpha_i$'s are nonzero, as the $\calD_i$'s are linearly equivalent and move in a one-parameter family by \Cref{lem:divisors_Di_in_case_2}.

We therefore see that the monomial $m=s^{2k}\lambda^{2k}$ appears with nonzero coefficient in the expansion of $\prod_{i=1}^k g_i^2$. This monomial can only divide the coordinate functions $f_0,\ldots,f_{d-2k}$. 
Let $f'_i=f_i/m$ for each such function divisible by $m$. Then $h$ must be of the form
\[h=u_0f'_0+\cdots+u_{d-2k}f'_{d-2k},\]
for some coefficients $u_i$ not all zero (if $f_i$ is not divisible by $m$, then consider $u_i=0$ already). Note that $f_i'$ is divisible by $s^{d-2k-i}$. But if we now consider
\[h\cdot\prod_{i=1}^k(\beta_i\mu^2+\gamma_i\lambda\mu)^2=u_0f_0'\prod_{i=1}^k(\beta_i\mu^2+\gamma_i\lambda\mu)^2+\cdots+u_{d-2k}f'_{d-2k}\prod_{i=1}^k(\beta_i\mu^2+\gamma_i\lambda\mu)^2\]
and recall that each coordinate function is divisible by a distinct power of $s$, we see that we must have that
\[u_{d-2k}f'_{d-2k}\prod_{i=1}^k(\beta_i\mu^2+\gamma_i\lambda\mu)^2=vf_d\]
for some $v\in\CC$. If the conics $\V(\beta_i\mu^2+\gamma_i\lambda\mu)\subseteq\PP^2$ are moving, then we see that $u_{d-2k}=0$ since $f_d$ is fixed. Otherwise, there exists $\overline{\beta},\overline{\gamma}\in\CC$ such that $\beta_i=\overline{\beta}$ and $\gamma_i=\overline{\gamma}$ for all $i$, and so the left-hand side with $\mu=1$ has a root with multiplicity $2k>1$. However, $p_{d-1}(\lambda,1)$ and therefore the right-hand side only has simple roots. Thus, we must have that $u_{d-2k}=0$. We can repeat this argument with $u_{d-2k-1},\ldots,u_0$ in this order to see that all of the $u_i$'s must be zero.

\smallskip
\noindent
\textbf{The case $a=1$:} Again, we can assume that $b_i \leq 1$. In fact, since $D$ is moving, at most one of the coefficients  in the expansion \eqref{eq:D-expansion} of $D$ can be nonzero (if any two coefficients were equal to one, this would fix either two points or a point and a direction of $D$).  

We begin with the case $b_1=1$ (meaning that the representatives of $D$ are strict transforms of lines passing through $P_1$). 
Then the projection of the lines $\calD_i$ must come from linear forms of the form 
$g_i(\lambda,\mu,s) = \alpha_i \lambda +\beta_i \mu$, where $(\alpha_i,\beta_i)\neq (0,0)$. It then follows from \Cref{lem:divisors_Di_in_case_2} that there exists a linear combination $f_\gen=\sum_{i=0}^d a_if_i$ for generic coefficients $a_0,\ldots,a_d$, and a nonzero homogeneous polynomial $h(\lambda,\mu,s)$ such that
\begin{equation}\label{eq2:F-A}
    f_\gen= h(\lambda,\mu,s) \cdot \prod_{i=1}^k \left( \alpha_i \lambda+\beta_i \mu\right)^2. 
\end{equation} 
We now give a divisibility argument similar to that in the $a=2$ case to show that this is impossible. 
First, if $\alpha_i=0$ for all $i$, the projections of the lines $\calD_i$ would pass through both $P_1$ and $P_2$. Then both $b_1,b_2\neq 0$ in \eqref{eq:D-expansion}, contradicting the assumption that at most one of the coefficients is nonzero. We may therefore assume that all $\alpha_i\neq 0$. Then $\lambda^{2k}$ appears with nonzero coefficient in the expansion of $\prod_{i=1}^k g_i^2$, and this monomial only divides the coordinate functions $f_0,\ldots,f_{d-2k}$. Then, by the same argument as in the $a=2$ case, this implies that the polynomial $h(\lambda,\mu,s)$ must be zero, which is a contradiction.

The case $b_2=1$ (meaning that all representatives of $D$ are strict transforms of lines passing through $P_2$) is similar (we instead get each $\calD_i$ would  come from linear forms in $\mu$ and $s$).

Finally, we consider the possibility that $b_1=b_2=0$. In this case,
\[A=H-2kD=(2d-1-2k)L-(d-1)E_{1,1}-dE_2-(d-1)E_{1,2}-\cdots.\]
We show that the curve $\calA$ on $\calS_d$ must have a number of irreducible components that will eventually contradict the condition $d\leq 3k+2$. 

Consider three curves on $\calS_d$: the 
strict transform of the line $\mathcal{L}_1$ between $P_1$ and $P_2$, the strict transform of the line $\mathcal{L}_2$ through $P_1$ whose strict transform passes through $P_{1,2}$, and the strict transform of the exceptional divisor $\mathcal{E}_{1,1}$ over $P_1$. By \Cref{lem:strict_transforms_gamma}, the class of the first in $\Pic(\calS_d)$ is $L-E_{1,1}-E_2$, of the second is $L-E_{1,1}-E_{1,2}$ and the third is $E_{1,1}-E_{1,2}.$
Now, since we have a negative intersection multiplicity  
$$A\cdot (L-E_{1,1}-E_2)=(2d-1-2k)-(d-1)-d=-2k-1<0,$$ 
we conclude that the strict transform of $\mathcal{L}_1$ is a component of $\calA$, appearing with some multiplicity $m>0$. Then the rest of $\calA$ has class
\[A-m(L-E_{1,1}-E_2)=(2d-1-2k-m)L-(d-1-m)E_{1,1}-(d-m)E_2-(d-1)E_{1,2}-\cdots\]
and has intersection multiplicity with the strict transform of $\mathcal{L}_2$ given by
\[(A-m(L-E_{1,1}-E_2))\cdot (L-E_{1,1}-E_{1,2})=(2d-1-2k-m)-(d-1-m)-(d-1)=1-2k<0.\] This is negative, so this strict transform is a component of $\calA$ with some multiplicity $m'.$
So the new rest of $\calA$ has class
\begin{align*}A-&m(L-E_{1,1}-E_2)-m'(L-E_{1,1}-E_{1,2})=\\
&(2d-1-2k-m-m')L-(d-1-m-m')E_{1,1}-(d-m)E_2-(d-1-m')E_{1,2}-\cdots.
\end{align*}
Finally, this has intersection multiplicity with the class $E_{1,1}-E_{1,2}$ given by
\[(A-m(L-E_{1,1}-E_2)-m'(L-E_{1,1}-E_{1,2}))\cdot (E_{1,1}-E_{1,2})=(d-1-m-m')-(d-1-m')=-m<0\]
so the corresponding strict transform appears with some multiplicity $m''>0$ in $\calA$.  The rest $\calA'$ of $\calA$, after subtracting the three kinds of components we have found, has class
\begin{align*}
A'&=A-m(L-E_{1,1}-E_2)-m'(L-E_{1,1}-E_{1,2})-m''(E_{1,1}-E_{1,2})\\
&=(2d-1-2k-m-m')L-(d-1-m-m'+m'')E_{1,1}-(d-m)E_2\\
&-(d-1-m'-m'')E_{1,2}-\cdots.
\end{align*}
Now, $$A'\cdot (E_{1,1}-E_{1,2})=d-1-m-m'+m''-(d-1-m'-m'')=2m''-m\geq 0$$ only if $m''\geq m/2.$ Furthermore $$A'\cdot (L-E_{1,1}-E_{1,2})=1-2k+m'\geq 0$$ only if $m'\geq 2k-1$, and 
$$A'\cdot (L-E_{1,1}-E_2)=-2k-m-m'-(-m-m'+m'')-(-m)=-2k+m-m''\geq 0$$ only if $m-m''\geq 2k$.
But $m''\geq m/2$ and $m-m''\geq 2k$ means $m\geq 4k$, so when in addition  $m'\geq 2k-1$, we get 
$$0\leq A'\cdot L=2d-1-2k-m-m'\leq 2d-1-2k-4k-2k+1=2d-8k,$$ so $d\geq 4k$.  
So we conclude that $4k\leq d\leq 3k+2$, which is possible only if $k\leq 2$ and $d\leq 8$, which is covered by \Cref{lem:base_cases}. 
\end{proof}

\section{The gamma distribution}\label{sec:gamma}
Similarly to the inverse Gaussian, the gamma distribution is given by two parameters: a shape parameter $k$ and a scale parameter $\theta$. Expressed in these parameters, the $r$th moment is 
\[m_r=\theta^r\prod_{i=0}^{r-1}(i+k)\quad \text{for $r\geq 0$}.\]
We have the following determinantal realization and description of the singular locus of the $d$th moment variety $\M_d^\Gamma\subseteq\PP^d$.

\begin{theorem}[{§4, \cite{HSY24}}]\label{thm:gamma_results_from_old_paper}
Let $d \geq 3$. The homogeneous prime ideal of the gamma moment variety $\M^\Gamma_d$ is  generated by the $\binom{d}{3}$ cubics given by the maximal minors of the $(3\times d)$-matrix
$$\begin{pmatrix}0  & x_1 & 2x_2 & 3x_3&\cdots& (d-1)x_{d-1}\\
x_0 & x_1 & x_2 & x_3 &\cdots& x_{d-1}\\
x_1 & x_2 & x_3 & x_4  & \cdots& x_d
\end{pmatrix}.$$
Furthermore, $\M^\Gamma_d$ has degree $\binom{d}{2}$. The singular locus is given by two points in $\PP^d$:
$$x_0=x_1=\cdots=x_{d-1}=0\quad \text{and}\quad x_1=x_2=\cdots=x_d=0.$$
\end{theorem}

As in \Cref{sec:inverse_gaussian}, the moment variety can be parametrized by homogenizing the moments, to obtain the rational map
\[\PP^d\dashrightarrow\M_d^\Gamma,\quad [\theta,k,s]\mapsto \big[s^{2(d-1)}\theta^r\prod_{i=0}^{r-1}(is+k):r=0,\ldots,r-1\big].\]
However, it turns out that resolving the indeterminacy locus becomes simpler after the coordinate change $x=s^2$, $y=\theta s$ and $z=\theta k$. In the new coordinates, we have the parametrization
\[\phi\from \PP^2\dashrightarrow\M_d^\Gamma\subseteq\PP^d,\quad [x:y:z]\mapsto [f_0(x,y,z):\cdots:f_d(x,y,z)],\]
where the $r$th coordinate map is given by
\[f_r(x,y,z)=x^{d-r}\prod_{i=0}^{r-1}(z+iy).\]
In the language of \cite{ARS18}, these polynomials are the maximal minors of the Hilbert--Burch matrix
$$\begin{pmatrix}
    z & x & 0&0&\cdots&0&0\\
    0& y+z& x& 0& \cdots&0&0\\
    0&0&2y+z & x& \cdots&0&0\\
    \vdots&\vdots&\vdots&\vdots&\ddots&\vdots&\vdots\\
    0&0&0&0&\cdots& (d-1)y+z & x\\
\end{pmatrix}.$$

The parametrization $\phi\from\PP^2\dashrightarrow\M_d^\Gamma$ has finitely many indeterminacy points $P_i=[0:1:-i]$ for $i=0,\ldots,d-1$. 
We now proceed with the strategy outlined in \Cref{sec:strategy} by constructing a resolution $\pi\from\calS_d\to\PP^2$ of the indeterminacy locus.
We end up needing $\ell_i=d-i$ blowups over the $i$th indeterminacy point, which we prove in the following lemma. 

The expressions in the proof are quite involved, and we refer the reader to \Cref{eg:gamma blow up d=4} below for explicit formulations in the $d=4$ case. The intersection-theoretic consequences of the construction of $\calS_d$ that will be used in the rest of the section are gathered in \Cref{lem:strict_transforms_gamma}.

\begin{lemma}\label[lem]{lem:gamma_blowup}
    Fix $i\in\{0,\ldots,d-1\}$, and let $j\in\{1,\ldots,d-i\}$. Suppose we have blown up at $P_{i,0},\ldots,P_{i,j-1}$ to obtain the surface $\calS_{i,j}\to\PP^2$, along with the exceptional divisor $\E_{i,j}$ and lift $\phi_{i,j}\from\calS_{i,j}\dashrightarrow\mathcal{M}_d$ of $\phi\from\PP^2\dashrightarrow\mathcal{M}_d$.
    Let $f_\gen=\sum_{k=0}^d a_kf_k$ be a linear combination of the coordinate functions $f_k$ for $\phi$, with general coefficients $(a_0,\ldots,a_d)\in\CC^{d+1}$.
    Then the following hold:
    \begin{enumerate}
        \item Let $j=1$. If $i<d-1$, then the lift $\phi_{i,1}\from\calS_{i,1}\dashrightarrow\PP^d$ has a single new indeterminacy point $P_{i,1}$ on $\E_{i,1}$, which does not lie on the strict transform of the line $\V(x)$, and which is a point of multiplicity 1 on the strict transform of the curve $\V(f_\gen)$. If $i=d-1$, then there are no new indeterminacy points on $\E_{d-1,1}$.
        \item If $1<j\leq d-i-1$, then the lift $\phi_{i,j}\from\calS_{i,j}\dashrightarrow\PP^d$ has a single new indeterminacy point $P_{i,j}$ on $\E_{i,j}$, which does not lie on the strict transform of $\E_{i,j-1}$, and which is a point of multiplicity one on the strict transform of the curve $\V(f_\gen)$.
        \item If $i<d-1$ and $j=d-i$, then the lift $\phi_{i,d-i}\from\calS_{i,d-i}\dashrightarrow\PP^d$ has no new indeterminacy points on $\E_{i,d-i}$.
    \end{enumerate}
\end{lemma}

\begin{proof}
\textbf{Part (1):} Let $j=1$. Consider the affine chart $\PP^2\cap\{y=1\}\cong\AA^2_{(x,z)}$. Blowing up at $P_{i,0}$ corresponds to blowing up at $(0,-i)$ in this chart. The resulting blowup is given by
\[\calB_{i,1}=\{((x,z),[u_1\mathbin{:}v_1])\in\AA^2\times\PP^1:xv_1=(z+i)u_1\},\]
with blowup morphism given by projection onto $\AA^2$. Recall that $\phi_{i,1}\from\calB_{i,1}\dashrightarrow\PP^d$ denotes the lift of $\phi$.
In $\calB_{i,1}$, the exceptional divisor $\E_{i,1}$ is given by $\{(0,-i)\}\times\PP^1$,
and the strict transform of the line $\V(x)$ is given by $\{((0,z),[0\mathbin{:}1]):z\in\CC\}$.

Consider the affine chart $\calB_{i,1}\cap\{u_1=1\}\cong\AA^2_{(x,v_1)}$, where the isomorphism is given by $(x,z)\mapsto (x,xv_1-i)$. 
The restriction to $\AA^2_{(x,v_1)}\dashrightarrow\PP^d$ of $\phi_{i,1}$, after factoring out the common factor $x$, is given by
\begin{align*}
(x,v_1)\mapsto \left[\begin{array}{l}x^{d-1}:x^{d-2}(xv_1-i):\cdots\\[0.2em]
x^{d-k-1}(xv_1-i)(xv_1-i+1)\cdots(xv_1-i+(k-1)):\cdots\\[0.2em]
v_1(xv_1-i)(xv_1-i+1)\cdots(xv_1-1)(xv_1+1)\cdots(xv_1+d-i-1)\end{array}\right].
\end{align*}
The first $d-1$ coordinate functions of $\phi_{i,1}$ (so $k\leq d-2$) always have $x$ as a factor, and the last coordinate function always has $v_1$ as a factor. If $i=d-1$, the $d$th function, i.e., when $k=d-1$, is not divisible by either $x$ or by $v_1$. Then $\phi_{d-1,1}$ has no indeterminacy points. When $i\neq d-1$, the $d$th function also has $v_1$ as a factor, so there is a indeterminacy point $P_{i,1}=((0,-i),[1:0])$ on $\E_{i,1}$. 

Since the smallest-degree monomial in $x$ and $v_1$ in the last coordinate function is linear, it follows that $P_{i,1}$ is a point of multiplicity one on the strict transform of the curve $\V(f_\gen)$.

Now consider the affine chart $\calB_{i,1}\cap\{v_1=1\}\cong\AA^2_{(z,u_1)}$, where the isomorphism is given by $(x,z)\mapsto ((z+i)u_1,z)$. Then the restriction to $\AA^2_{(z,u_1)}\dashrightarrow\PP^d$ of $\phi_{i,1}$, after factoring out the common factor $(z+i)$,  is given by
\begin{align*}
(z,u_1)\mapsto&\left[\begin{array}{l}(z+i)^{d-1}u_1^d:\cdots\\[0.2em]
(z+i)^{d-k-1}u_1^{d-k}z(z+1)\cdots(z+(k-1)):\cdots\\[0.2em]
z(z+1)\cdots(z+(i-1))(z+(i+1))\cdots (z+(d-1))\end{array}\right],    
\end{align*}
which has no indeterminacy points on the exceptional divisor $\E_{i,1}\cap\{v_1=1\}$.
The indeterminacy points $((0,-i'),[0:1])$, for $i'\in\{0,\ldots,d-1\}\setminus\{i\}$, do not lie on the exceptional divisor $\E_{i,1}$, and are simply the preimages of the original indeterminacy points $P_{i',0}$.  
Thus, claim (1) holds.

\smallskip
\noindent
\textbf{Part (2) and (3):} Suppose $j\in\{2,\ldots,d-i-1\}$, and that we have blown up at $P_{i,0},\ldots,P_{i,j-2}$ to obtain $\calS_{i,j-1}\to\PP^2$, and lift $\phi_{i,j-1}\from \calS_{i,j-1}\dashrightarrow\PP^d$ of $\phi$. Let $\calB_{i,j-1}\subseteq\AA^2\times\PP^1$ denote the local chart of the blowup around $P_{i,j-1}$ from the previous step of the construction.

By induction, blowing up at $P_{i,j-1}$ corresponds to blowing up at the origin $(0,0)$ in the chart $\calB_{i,j-1}\cap\{u_{j-1}=1\}\cong\AA^2_{(x,v_{j-1})}$. In these coordinates, the resulting blowup $\calS_{i,j}$ is the closure of $\calB_{i,j}\subseteq\calS_{i,j-1}\times \PP^1$, where
\[\calB_{i,j}=\{((x,v_{j-1}),[u_{j}:v_{j}])\in\AA^2\times\PP^1:xv_{j}=v_{j-1}u_{j}\}.\]
The exceptional divisor $\E_{i,j}\subseteq\calB_{i,j}$ is $\{(0,0)\}\times\PP^1$. The strict transform of the previous exceptional divisor $\E_{i,j-1}$ is locally given by $\{((0,v_{j-1}),[0:1]):v_{j-1}\in\CC\}$.

Consider the affine chart $\calB_{i,j}\cap\{u_{j}=1\}=\AA^2_{(x,v_j)}$. Here, the lift $\phi_{i,j}\from\calS_{i,j}\dashrightarrow\PP^d$ of $\phi_{i,j-1}$, after substituting the appropriate coordinates and factoring $x$ from all coordinate functions, is given by $\phi_{i,j}(x,v_j)=[g_0:\cdots:g_d]$, where
\[g_k(x,v_j)=\begin{cases}
    x^{d-j-k}\prod_{-i\leq\ell\leq k-i-1}(x^{j}v_j+\ell)\quad& 0\leq k\leq i,\\\\
    x^{d-k}v_j\prod_{\substack{-i\leq \ell\leq k-i-1,\\\ell\neq 0}}(x^{j}v_j+\ell)\quad &i+1\leq k\leq d.
\end{cases}\]

If $i>0$ and $j=d-i$, i.e., $i=d-j$, then the $g_{d-j}$ coordinate function is of the form
\begin{align*}
    g_{d-j}(x,v_j)&=(x^jv_j-i)(x^jv_j-(i-1)\cdots (x^jv_j-i+(d-j-1))\\
    &=(x^jv_j-i)(x^jv_j-(i-1))\cdots(x^jv_j-1),
\end{align*}
which is not divisible by either $x$ or $v_j$. However, note that $g_0=x^{d-(d-i)}=x^i$, and
\[g_d=v_j\prod_{\substack{-i\leq \ell\leq d-i-1,\\\ell\neq 0}}(x^{j+1}v_j+\ell).\]
Thus, there are no indeterminacy points of $\phi_{i,j}$ in this case. If $i=0$ and $j=d$, then we have that $g_0=1$, and so there are no indeterminacy points in this case either.

If $2\leq j\leq d-i-1$, then there is a indeterminacy point $P_{i,j}=((0,0),[1:0])$ on $\E_{i,j}$, as the coordinate functions $g_0,\ldots,g_{d-j}$ are divisible by $x$, while the functions $g_{d-j},\ldots,g_d$ are divisible by $v_j$. Notice that $P_{i,j}$ does not lie on the strict transform of $\E_{i,j-1}$. Since the monomial in $x$ and $v_j$ of smallest degree of the last coordinate function $g_d$ is linear, it follows that $P_{i,j}$ is a point of multiplicity one on the strict transform of $\V(f_\gen)$ for a generic linear combination $f_\gen$ of the coordinate functions.

Now consider the affine chart $\calB_{i,j}\cap\{v_j=1\}=\AA^2_{(v_{j-1},u_j)}$. The lift $\phi_{i,j}$, described in local coordinates by making the appropriate substitutions and factoring $v_{j-1}$ from each of the coordinate functions, is given by $(v_{j-1},u_j)\mapsto[h_0:\cdots:h_d]$, where
\[h_k(v_{j-1},u_j)=\begin{cases}
    v_{j-1}^{d-j-k}u_j^{d-j-k+1}\prod_{-i\leq\ell\leq k-i-1}(v_{j-1}^{j}u_j^{j-1}+\ell)\quad& 0\leq k\leq i,\\\\
    (v_{j-1}u_j)^{d-k}\prod_{\substack{-i\leq \ell\leq k-i-1,\\\ell\neq 0}}(v_{j-1}^{j}u_j^{j-1}+\ell)\quad &i+1\leq k\leq d.
\end{cases}\]
Then, $h_0$ is a monomial in $v_{j-1}$ and $u_j$, while $h_d$ is not divisible by either variable. Therefore, there are no indeterminacy points of $\phi_{i,j}$ in this affine chart. This concludes the proof of the  claims (2) and (3) of the lemma.
\end{proof}

\begin{example}\label[ex]{eg:gamma blow up d=4}
For $d=4$, the rational map $\phi\from\PP^2\dashrightarrow\PP^4$ is given by
\[[x:y:z]\mapsto [x^{4} : x^3 z : x^2 z(z+y) : x z(z+y)(z+2y) : z(z+y)(z+2y)(z+3y)]\]
with indeterminacy points $P_0=[0:1:0]$, $P_1=[0:1:-1]$, $P_2=[0:1:-2]$ and $P_3=[0:1:-3]$. 
We will now demonstrate the proof of \Cref{lem:gamma_blowup} by resolving the singularity at $P_1$.

\smallskip
\noindent
\textbf{Blowing up at $P_1$:} We blow up at $(0,-1)$ in the chart $\PP^2\cap\{y\neq 0\}$ (isomorphic to $\AA^2$ via $(x,z)\mapsto [x:1:z]$), to obtain
\[\calB_{1,1}=\{((x,z),[u_1:v_1])\in\AA^2\times\PP^1:xv_1=(z+1)u_1\}.\]

The exceptional divisor is given by $\E_{1,1}=\{(0,-1)\}\times\PP^1$, and the strict transform of $\V(x)$ is given by $\{((0,z),[0:1]):z\in\AA^1\}$. In the $\calB_{1,1}\cap\{u_1\neq 0\}$ chart, which is isomorphic to $\AA^2$ via $(x,v_1)\mapsto ((x,xv_1-1),[1:v_1])$, the lift $\phi_{1,1}$ is generically given by
\[(x,v_1)\mapsto\left[\begin{array}{l}
x^{3}:
\\
 x^{2} \left(x v_{1}-1\right): 
\\
 x^{2} \left(x v_{1}-1\right) v_{1} :
\\
 x \left(x v_{1}-1\right) v_{1} \left(x v_{1}+1\right):
\\
 \left(x v_{1}-1\right) v_{1} \left(x v_{1}+1\right) \left(x v_{1}+2\right) 
\end{array}\right].\]
There is a indeterminacy point $(0,0)\in\AA^2$, which corresponds to $((0,-1),[1:0])\in\E_{1,1}$. On the other hand, in the chart $\calB_{1,1}\cap\{v_1\neq 0\}$, the lift $\phi_{1,1}$ is given by
\[(z,u_1)\mapsto \left[\begin{array}{l}
\left(z +1\right)^{3} u_{1}^{4} : \\
 \left(z +1\right)^{2} u_{1}^{3} z  :\\
 \left(z +1\right) u_{1}^{2} z \left(z +1\right)  :\\
 u_{1} z \left(z +1\right) \left(z +2\right) :\\
 z \left(z +2\right) \left(z +3\right) 
\end{array}\right]
\]
with no further indeterminacy points on $\E_{1,1}$. 

\smallskip
\noindent
\textbf{Blowing up at $P_{1,1}$:} We now construct $\calS_{1,2}$ by blowing up $\calB_{1,1}\cap\{u_1\neq 0\}\cong\AA^2$ in the origin, which gives 
\[\calB_{1,2}=\{((x,v_1),[u_2:v_2]\in\AA^2\times\PP^1:xv_2=v_1u_2\}.\]
The exceptional divisor is $\E_{1,2}=\{(0,0)\}\times\PP^1$, and the strict transform of $\E_{1,1}$ is $\{((0,v_1),[0:1]):v_1\in\AA^1\}$. In the chart $\calB_{1,2}\cap\{u_2\neq 0\}$, the lift $\phi_{1,2}$ is given by
\[(x,v_2)\mapsto\left[\begin{array}{l}
x^{2} : \\
 x (x^{2}v_{2}-1) : \\
 x^{2} (x^{2}v_{2}-1) v_{2} : \\
 x (x^{2}v_{2}-1) v_{2} (x^{2}v_{2}+1) : \\
 (x^{2}v_{2}-1) v_{2} (x^{2}v_{2}+1) (x^{2}v_{2}+2) 
\end{array}\right]
\]
with a new indeterminacy point $(0,0)\in\AA^2$ that corresponds to $((0,0),[1:0])\in\E_{1,2}$. On the other hand, in the chart $\calB_{1,2}\cap\{v_2\neq 0\}$, the lift is given by
\[(z,v_2)\mapsto \left[\begin{array}{l}
u_{2}^{3} v_{1}^{2} : \\
 v_{1} u_{2}^{2} (v_{1}^{2} u_{2}-1) :\\
 v_{1}^{2} u_{2}^{2} (v_{1}^{2} u_{2}-1) :\\
 v_{1} u_{2} (v_{1}^{2} u_{2}-1) (v_{1}^{2} u_{2}+1) : \\
 (v_{1}^{2} u_{2}-1) (v_{1}^{2} u_{2}+1) (v_{1}^{2} u_{2}+2) 
\end{array}\right]
\]
and lacks further indeterminacy points.

\smallskip
\noindent
\textbf{Blowing up at $P_{1,2}$:} We now construct $\calS_{1,3}$ by blowing up $\calB_{1,2}\cap\{u_2\neq 0\}\cong\AA^2$ in the origin, which gives 
\[\calB_{1,3}=\{((x,v_2),[u_3:v_3]\in\AA^2\times\PP^1:xv_3=v_2u_3\}.\]
The exceptional divisor is $\E_{1,3}=\{(0,0)\}\times\PP^1$, and the strict transform of $\E_{1,2}$ is $\{((0,v_2),[0:1]):v_2\in\AA^1\}$. In the chart $\calB_{1,3}\cap\{u_3\neq 0\}$, the lift $\phi_{1,3}$ is given by
\[(x,v_3)\mapsto\left[\begin{array}{l}
x : \\
 x^{3} v_{3}-1 : \\
 x^{2} (x^{3} v_{3}-1) v_{3} :\\
 x (x^{3} v_{3}-1) v_{3} (x^{3} v_{3}+1) :\\
 (x^{3} v_{3}-1) v_{3} (x^{3} v_{3}+1) (x^{3} v_{3}+2) 
\end{array}\right]
\]
and lacks further indeterminacy points. Similarly, in the chart $\calB_{1,3}\cap\{v_3\neq 0\}$, the lift is given by
\[(z,v_3)\mapsto \left[\begin{array}{l}
v_{2} u_{3}^{2} : \\
 u_{3} (u_{3}^{2} v_{2}^{3}-1) : \\
 v_{2}^{2} u_{3}^{2} (u_{3}^{2} v_{2}^{3}-1) : \\
 v_{2} u_{3} (u_{3}^{2} v_{2}^{3}-1) (u_{3}^{2} v_{2}^{3}+1) : \\
 (u_{3}^{2} v_{2}^{3}-1) (u_{3}^{2} v_{2}^{3}+1) (u_{3}^{2} v_{2}^{3}+2) 
\end{array}\right]\]
and lacks further indeterminacy points.
\end{example}

The following intersection-theoretic formulas follow directly from \Cref{lem:gamma_blowup} and \eqref{eq:class_of_strict_transform}.

\begin{lemma}\label[lem]{lem:strict_transforms_gamma}
Let $E_{i,j}$ denote the class in $\Pic(\calS_d)$ of the pullback of $\E_{i,j}$ to $\calS_d$ along the composition of appropriate blowup maps, and let $L$ denote the class of the pullback $\mathcal{L}$ of a line in $\PP^2$. The following formulas hold:
\begin{enumerate}
    \item The class of the strict transform of the line $\V(x)$ is given by $L-\sum_{i=0}^{d-1}E_{i,1}$.
    \item The class of the strict transform of $\E_{i,j}$ is given by $E_{i,j}-E_{i,j+1}$ for $i\in\{0,\ldots,d-1\}$ and $j\in\{1,\ldots,d-i-1\}$.
    \item The class of the strict transform of $\E_{i,d-i}$ is given by $E_{i,d-i}$ for $i\in\{0,\ldots,d-1\}$.
    \item Let $H$ be the class of the strict transform $\mathcal{H}$ of $\V(f_\gen)\subseteq\PP^2$ for a generic linear combination $f_\gen$ of the coordinate functions of $\phi$. Then
    \begin{equation}\label{eq:H_class_gamma}
    H=dL-\sum_{i=0}^{d-1} \sum_{j=1}^{d-i} E_{i,j}.
    \end{equation}
\end{enumerate}
\end{lemma}

As in the previous section, \Cref{lem:gamma_blowup,lem:strict_transforms_gamma} imply the following proposition.

\begin{proposition}
    The map $\tilde{\phi}\from \calS_d\to\M^\Gamma_d$ is a birational morphism. 
\end{proposition}
\begin{proof}
    On the one hand, by \Cref{lem:gamma_blowup}, the map $\tilde{\phi}:=\phi_{d-1,1}$ has no indeterminacy points, i.e., it is a morphism.
 On the other hand, by 
 \Cref{lem:strict_transforms_gamma}, we have that
\[H^2=d^2-d-(d-1)-\cdots-1=\binom{d}{2}=\deg(\M_d^{\Gamma}),\]
where the last equality follows from \Cref{thm:gamma_results_from_old_paper}. 
Since $H$ is the pullback of a hyperplane section of $\M_d^{\Gamma}$ along $\tilde{\phi}$, we have that  $H^2=\deg(\M_d^\Gamma)\deg(\tilde{\phi})$. Therefore, we conclude that $\deg(\tilde{\phi})=1$. Since $\tilde{\phi}$ is dominant, the desired result follows.
\end{proof}

We are now ready to state and prove the main theorem of this section.

\begin{theorem}\label{thm:gamma_nondef}
The moment variety $\M_d^\Gamma$ is $k$-nondefective for all $k\geq 2$ and $d\geq 2$.
\end{theorem}

\begin{proof}
We will use the same proof strategy as for the inverse Gaussian distribution (\Cref{thm:IG_nondef}). Fix $d$, and assume that $\M_d^\Gamma$ is $k$-defective. Then case (2) of \Cref{thm:terracini} must apply, and so by \Cref{lem:divisors_Di_in_case_2}, we have an expression for the class $H$ of the pullback of a general hyperplane section $\mathcal{H}$ of $\M_d^\Gamma$ to $\calS_d$.

In particular, there exist linearly equivalent divisors $\calD_1,\ldots,\calD_k$ of $\calS_d$ with class
\[D=aL-\sum_{i=0}^{d-1}\sum_{j=1}^{d-i} b_{i,j}E_{i,j}\]
with coefficients $a>0$ and $b_{i,j}\geq 0$, such that   
$$A=H-2kD=(d-2ka)L-\sum_{i=0}^{d-1}\sum_{j=1}^{d-i} (1-b_{i,j})E_{i,j}$$ 
is the class of an effective divisor by \Cref{lem:strict_transforms_gamma}. As in the inverse Gaussian case, we proceed by casework on the value of $a$, which is the degree of the projections $\pi(\calD_i)$ to $\PP^2$. We will almost immediately be able to reduce to the $a=1$ case, and derive a contradiction there.

We have that $A\cdot L=d-2ka\geq 0$. 
Together with \Cref{lem:lower_bound_on_d_for_defective_surfaces}, we obtain the inequalities
\begin{equation}
    2ak \leq d \leq 3k+2.
\end{equation} 
This immediately allows us to rule out the case $a\geq 3$ (since that would give $6k\leq 3k+2$ which is impossible for $k\geq 2$). 

In the case $a=2$, we get $6k\leq 3k+2$ and the only possibility is $k=2$ and $d\leq 8$, which is covered by \Cref{lem:base_cases}.

Therefore, we are left to investigate the $a=1$ case. Here, we have
$D=L-\sum_{i=0}^{d-1}\sum_{j=1}^{d-i} b_{i,j}E_{i,j}$, and since $D$ is moving, at most one of the coefficients $b_{i,1}$ is 1, and the coefficients of the classes of the other exceptional divisors are 0. We will now investigate each of the possibilities, and see that each of them leads to a contradiction.

\smallskip
\noindent
\textbf{The case $b_{i,1}=1$ for some $i\in\{0,\ldots,d-1\}$:} This means that for each $\ell\in\{1,\ldots,k\}$, the curve $\calD_\ell$ is the strict transform of a line in $\PP^2$ passing through $P_i$, which means that it is given by a linear form $$g_\ell(x,y,z)=\alpha_\ell x+\beta_\ell (iy + z)$$
for some $(\alpha_\ell,\beta_\ell)\neq (0,0)$. Note that $g_\ell(0,1,-i)=0$. Hence, there is a linear combination $f_\gen=\sum_{j=0}^d a_jf_j$ with $(a_0,\ldots,a_d)\neq 0$, such that 
\begin{equation}\label{eq:F_divisible_by_linear_forms_gamma}    f_\gen=h(x,y,z)\prod_{\ell=1}^k(\alpha_\ell x+i\beta_\ell y +\beta_\ell z)^2
\end{equation} for a homogeneous polynomial $h(x,y,z)\neq 0$. From this we can derive a contradiction by a similar divisibility argument as we used for the inverse Gaussian case. 

The idea is as follows: We have that $f_\gen$ is a homogeneous polynomial of degree $d$, and that $(g_1\cdots g_k)^2$ is homogeneous of degree $2k$. Thus, $h$ is homogeneous of degree $d-2k\geq 0$. Note that $f_j$ only involves monomials of the form $x^{d-j}y^rz^s$, so no monomial appears in more than one of the $f_j$'s. Furthermore, recall that if $P_i$ is a root of a coordinate function $f_j$, then it is a simple root.

If $\alpha_\ell=0$ for some $\ell$, then $\alpha_\ell=0$ for all $\ell$, and so $f_\gen$ has a root at $P_i$ with multiplicity at least $2k$. But this is not possible, as this would imply that any coordinate functions $f_j$ appearing in $f_\gen$ also have $P_i$ as a root with multiplicity at least $2k>1$, contradicting the fact that the $P_i$'s are simple roots of the coordinate functions. Hence, we conclude that the monomial $x^{2k}$ appears as a monomial in $(g_1\cdots g_k)^2$. However, it only divides $f_0, \ldots, f_{d-2k}$ in the left hand side of \eqref{eq:F_divisible_by_linear_forms_gamma}. Thus, we can apply a similar argument to the one applied for the inverse Gaussian in the $a=2$ case, to conclude that $h(x,y,z)=0$, which is a contradiction.

\smallskip
\noindent
\textbf{The case $b_{i,1}=0$ for all $i\in\{0,\ldots,d-1\}$:} In this case, $A$ has the following form:
\[A=(d-2k)L-\sum_{i=0}^{d-1}\sum_{j=1}^{d-i} E_{i,j}.\]
Note that $d-2k>0$ since $\calA$ is effective. We will derive a contradiction by showing that $\calA$ has the irreducible effective divisors from \Cref{lem:strict_transforms_gamma} as components, and that after removing them in a certain order, we obtain a divisor that is not effective, thereby contradicting the effectiveness of $\calA$.

First, we have that
\[A\cdot \Big(L-\sum_{i=0}^{d-1}E_{i,1}\Big)=(d-2k)-d=-2k<0,\]
and so $L-\sum_{i=0}^{d-1}E_{i,1}$ is the class of a divisor that is a component of $\calA$. Let $A_1=A-(L-\sum_{i=0}^{d-1}E_{i,1})$ be the class of the divisor $\calA_1$ obtained by removing this component, so
\[A_1=(d-2k-1)L-\sum_{i=0}^{d-2}\sum_{j=2}^{d-i}E_{i,j}.\]
We have that $A_1\cdot (E_{i,1}-E_{i,2})=-1$ for $i=0,\ldots,d-2$, so we can remove each of the components corresponding to these classes $E_{i,1}-E_{i,2}$ to obtain a divisor $\mathcal{A}'_1$, with class in the Picard group
\begin{align*}
    A'_1&=A_1-(E_{0,1}-E_{0,2})-\cdots-(E_{d-2,1}-E_{d-2,2})\\
    &=(d-2k-1)L-\sum_{i=0}^{d-2}E_{i,1}-\sum_{i=0}^{d-3}\sum_{j=3}^{d-i}E_{i,j}.
\end{align*}
Then, 
\[A'_1\cdot \Big(L-\sum_{i=0}^{d-1}E_{i,1}\Big)=(d-2k-1)-(d-1)=-2k<0,\] 
and so we can once again remove $L-\sum_{i=0}^{d-1}E_{i,1}$ to obtain
\[A_2=A'_1-\Big(L-\sum_{i=0}^{d-1}E_{i,1}\Big)=(d-2k-2)L+E_{d-1,1}-\sum_{i=0}^{d-3}\sum_{j=3}^{d-i}E_{i,j}.\]
We have that $A_2\cdot (E_{i,2}-E_{i,3})=-1$ for $i=0,\ldots,d-3$, so we can remove each of these components to obtain
\begin{align*}
    A'_2&=A_2-(E_{0,2}-E_{0,3})-\cdots-(E_{d-3,2}-E_{d-3,3})\\
    &=(d-2k-2)L+E_{d-1,1}-\sum_{i=0}^{d-3}E_{i,2}-\sum_{i=0}^{d-4}\sum_{j=4}^{d-i}E_{i,j}.
\end{align*}
We then have that $A'_2\cdot(E_{i,1}-E_{i,2})=-1$ for $i=0,\ldots,d-3$, so we can also remove each of the components corresponding to these classes to obtain
\begin{align*}
    A''_2&=B_2-(E_{0,1}-E_{0,2})-\cdots-(E_{d-3,1}-E_{d-3,2})\\
    &=(d-2k-2)L+E_{d-1,1}-\sum_{i=0}^{d-3}E_{i,1}-\sum_{i=0}^{d-4}\sum_{j=4}^{d-i}E_{i,j}.
\end{align*}
Finally, we have that $A''_2\cdot E_{d-1,1}=-1$, so removing the corresponding components, we obtain
\[A'''_2=(d-2k-2)L-\sum_{i=0}^{d-3}E_{i,1}-\sum_{i=4}^{d-4}\sum_{j=4}^{d-i}E_{i,j}.\]
Now, $A''_2\cdot (L-\sum_{i=0}^{d-1}E_{i,1})=(d-2k-2)-(d-2)=-2k<0$, so we can remove the corresponding component to obtain
\[A_3=(d-2k-3)L+E_{d-2,1}+E_{d-1,1}-\sum_{i=0}^{d-4}\sum_{j=4}^{d-i}E_{i,j}.\]
We continue in this way, removing effective irreducible components. Eventually, we have removed $d-2k-1$ copies of $\mathcal{L}$, along with many other effective divisors, and have the divisor $\calA_{d-2k-1}$ with class
\begin{align*}
    A_{d-2k-1}&=(d-2k-(d-2k-1))L+\sum_{i=d-(d-2k-2)}^{d-1}E_{i,1}-\sum_{i=0}^{d-(d-2k-1)-1}\sum_{j=d-2k}^{d-i}E_{i,j}\\
    &=L+\sum_{i=2k+2}^{d-1}E_{i,1}-\sum_{i=0}^{2k}\sum_{j=d-2k}^{d-i}E_{i,j}.
\end{align*}
Similarly, the divisors with the following classes are components of $\calA_{d-2k-1}$, and they can be removed in the listed order:
\begin{align*}
    &E_{0,d-2k-1}-E_{0,d-2k},\ldots,E_{2k,d-2k-1}-E_{2k,d-2k},\\
    &E_{0,d-2k-2}-E_{0,d-2k-1},\ldots,E_{2k,d-2k-2}-E_{2k,d-2k-1},\\
    &\vdots\\
    &E_{0,1}-E_{0,2},\ldots,E_{2k,1}-E_{2k,1},\\
    &E_{2k+1,1},\ldots,E_{d-1,1}.
\end{align*}
If $d-2k=1$, then we are only removing the last line of divisors. We then obtain the divisor $\mathcal{A}'_{d-2k-1}$ with class
\[A'_{d-2k-1}=L-\sum_{i=0}^{2k}E_{i,1}-\sum_{i=0}^{2k-1}\sum_{j=d-2k+1}^{d-i}E_{i,j}.\]
Intersecting this with the strict transform of $\V(x)$, we see that
\[A'_{d-2k-1}\cdot \Big(L-\sum_{i=0}^{d-1}E_{i,1}\Big)=1-(2k+1)=-2k<0,\]
so the strict transform of $\mathcal{V}(x)$ is still a component, and we remove it to get the divisor $\mathcal{A}_{d-2k}$ with class
\[A_{d-2k}=A'_{d-2k-1}-\Big(L-\sum_{i=0}^{d-1}E_{i,1}\Big)=\sum_{i=2k+1}^{d-1}E_{i,1}-\sum_{i=0}^{2k-1}\sum_{j=d-2k+1}^{d-i}E_{i,j}.\]
But $\mathcal{A}_{d-2k}$ is not effective: the divisors whose classes have negative coefficients do not lie over those with positive coefficients. We have thus reached our contradiction.
\end{proof}

\section{Rational identifiability}\label{sec:rational_identifiability}

In the previous sections we have seen that the parameters of $k$-mixtures of the inverse Gaussian or gamma distribution are algebraically identifiable from the first $3k-1$ moments. A natural question is how many further moments we need in order to have \textit{rational identifiability}, in the sense that for generic sample moments for which there is a solution to the moment equations, the solution is unique up to the label swapping symmetry. In the language of algebraic geometry, this corresponds to the problem of \emph{$k$-identifiability}: Given a $k$, for what $d$ does it hold that a generic point of $\Sec_k(\M_d)$ lies on a unique $k$-secant? 

Based on numerical experiments, we conjecture that $d\geq 3k$ suffices, but in what follows we will instead prove the more modest claim that $d\geq 3k+2$ suffices. Our proof strategy will be the same as the one used in \cite[Section~3]{LAR25} to prove the analogous statement for the Gaussian distribution, namely, to use the following sufficient conditions for $k$-identifiability.

{\samepage
\begin{theorem}[{\cite[Theorem~1.5]{MM24}}]
    Let $X\subseteq\PP^d$ be an irreducible nondegenerate variety, and let $k\geq 2$. Then a generic point of $\Sec_k(X)$ lies on a unique $k$-secant if the following conditions are satisfied:
    \begin{enumerate}
        \item $(k+1)\dim(X)+k\leq d$
        \item $X$ is $(k+1)$-nondefective
        \item The Gauss map of $X$ is nondegenerate.
    \end{enumerate}
\end{theorem}
}

\begin{theorem}\label{thm:rational_identifiability}
For $k$-mixtures of the inverse Gaussian or the gamma distribution, we have rational identifiability from the first $3k+2$ moments.
\end{theorem}

\begin{proof}
Let $d=3k+2$. It is immediate from the results from \cite{HSY24} that $\M_{d}$ is irreducible and nondegenerate, and that condition (1) is satisfied, whereas condition (2) follows from \Cref{thm:IG_nondef,thm:gamma_nondef}. 

To prove (3), we use a classical result about Gauss maps of surfaces (see, e.g., \cite[Theorem~4.3.6]{IL03}), which says that if the Gauss map of $\M_{d}$ is degenerate, then $\M_{d}$ is either a cone over a curve, or the tangential variety of a curve. It follows from the proofs of \Cref{thm:IG_nondef,thm:gamma_nondef} that $\M_d$ cannot be a cone over a curve. Assume now for a contradiction that $\M_{d}$ is equal to the tangential variety $\tau(C)$ for some curve $C$ in $\PP^d$. Since $\M_{d}$ is not contained in a plane, $C$ is not a plane curve. Hence, $C$ is contained in the singular locus of $\M_d$. 
(This follows from the general fact that a nonplanar curve is contained in the singular locus of its tangential variety; see, e.g., \cite{Piene81} for the case of space curve, from which the general case readily follows.) 
For the gamma distribution, the singular locus is zero-dimensional by \Cref{thm:gamma_results_from_old_paper}, so this is impossible. In the inverse Gaussian case, the singular locus is a line and a point by \Cref{thm:IG_results_from_previous_paper}, so this would imply that $C$ is a line, which in turn would imply that $\tau(C)$ is a line, which contradicts $\tau(C)=\M_d$. 
\end{proof}

\newlength{\bibitemsep}\setlength{\bibitemsep}{.2\baselineskip plus .05\baselineskip minus .05\baselineskip}
\newlength{\bibparskip}\setlength{\bibparskip}{0pt}
\let\oldthebibliography\thebibliography
\renewcommand\thebibliography[1]{
  \oldthebibliography{#1}
  \setlength{\parskip}{\bibitemsep}
  \setlength{\itemsep}{\bibparskip}
}


\vspace{1em}

\noindent {\bf Authors' addresses:}

\noindent 
Oskar Henriksson, University of Copenhagen \hfill{\tt oskar.henriksson@math.ku.dk}\\
Kristian Ranestad, University of Oslo \hfill {\tt ranestad@math.uio.no}\\
Lisa Seccia, University of Neuchâtel \hfill{\tt lisa.seccia@unine.ch}\\
Teresa Yu, University of Michigan \hfill {\tt twyu@umich.edu}

\end{document}